\documentclass[a4paper,12pt, psamsfonts, reqno]{amsart}
\usepackage{verbatim}
\usepackage{amsmath}
\usepackage{amssymb}
\usepackage{textcomp}
\usepackage{amsthm}
\usepackage{amsfonts}
\usepackage[all]{xy}
\usepackage{mathrsfs}
 \usepackage{graphicx}
 \usepackage{epstopdf}
 \usepackage{float}
 \usepackage{graphicx}
 \usepackage[T1]{fontenc}
\usepackage{hyperref}
\newtheorem{theo}{Th\'eor\'eme}[section]
\usepackage[ddmmyyyy]{datetime}

\newcommand{\stirlingii}{\genfrac{\{}{\}}{0pt}{}}

\def\ker{\mathop{\mathrm{ker}}\nolimits}

\def\Hilb{\mathop{\mathrm{Hilb}}\nolimits}

\def\rank{\mathop{\mathrm{rank}}\nolimits}

\def\int{\mathop{\mathrm{int}}\nolimits}
\def\tor{\mathop{\mathrm{tor}}\nolimits}
\def\im{\mathop{\mathrm{Im}}\nolimits}
\def\Id{\mathop{\mathrm{Id}}\nolimits}

\def\Sym{\mathop{\mathrm{Sym}}\nolimits}

\def\GL{\mathop{\mathrm{GL}}\nolimits}
\def\mod{\mathop{\mathrm{mod}}\nolimits}

\def\remk{\noindent\textit{Remark:~}}

\newtheorem{prop}[theo]{Proposition}
\newtheorem{def-prop}[theo]{Definition-Proposition}
\newtheorem{cor}[theo]{Corollary}

\newtheorem*{teo*}{Theorem}
\newtheorem{lemma}[theo]{Lemma}
\newtheorem{teo}[theo]{Theorem}
\newtheorem{definition}[theo]{Definition}

\newtheorem{example}[theo]{Example}

 \newcommand{\cM}{{\mathcal M}}

  \newcommand{\bC}{{\mathbb C}}
  \newcommand{\bF}{{\mathbb F}}
 \newcommand{\bN}{{\mathbb N}}

 \newcommand{\bQ}{{\mathbb Q}}
  \newcommand{\bH}{{\mathbb H}}

 \newcommand{\bT}{{\mathbb T}}
 \newcommand{\bZ}{{\mathbb Z}}

  \def\SL {\mathop{\mathrm{SL}}\nolimits}
  \def\val {\mathop{\mathrm{val}}\nolimits}
    
      \def\prim {\mathop{\mathrm{prim}}\nolimits}

   \def\Ann {\mathop{\mathrm{Ann}}\nolimits}

\newcommand{\sE}{{\mathscr E}}

\allowdisplaybreaks

\everymath{\displaystyle} 

\input xy
\xyoption{all}

\title[Torsions in Cohomology]{Torsions in Cohomology of $\SL_2(\bZ)$ and Congruence of Modular Forms}

\date{}


\begin{document}

\author{Taiwang Deng}

\begin{abstract}
We describe torsion classes in the first cohomology group of 
$\text{SL}_2(\mathbb{Z})$. In particular, we obtain generalized Dickson's 
invariants for p-power polynomial rings. Secondly, 
we describe torsion classes in the
zero-th homology group of $\text{SL}_2(\mathbb{Z})$ as a module over
the torsion invariants. As application, we obtain various
congruences between cuspidal forms of level one and Eisenstein series.
\end{abstract}

\address{Max Planc Institut f{\"{u}}r Mathematik, 
 Vivatsgasse 7, 53111 Bonn.}
\email{dengtw@mpim-bonn.mpg.de}

\maketitle

\tableofcontents

\section{Introduction}
Let $\Gamma=\SL_2(\bZ)/\{\pm \Id\}$ and $\mathbb{H}=\{z\in \bC: \Im(z)>0\}$. 
Let 
\[
X:=\Gamma\backslash \bH
\]
In this article, we investigate the 
torsion classes in $H^1(X, \tilde{\cM}_n)$
and $H^2_c(X, \tilde{\cM}_n)$, where 
the sheaf $\tilde{\cM}_n$ is induced by
the action of $\SL_2(\bZ)$ on the space
of homogeneous polynomials of degree $n$, denoted by $\cM_n$.

Fix a prime $p>3$.\footnote{We remark that most part of the 
results in this article remain valid for 
$p=2$ and $p=3$, we exclude them for two reasons: one is due to the fact that the Lemma \ref{lemma-Harder-exact-functor} fails for these primes, the other is for being less technical.} We consider the following generalized 
Dickson's invariants
\[
f_{1, \delta}=(X^pY-XY^p)^{p^{\delta-1}}, \quad f_{2, \delta}=(\frac{X^{p^2-1}-Y^{p^2-1}}{X^{p-1}-Y^{p-1}})^{p^{\delta-1}}.
\]
Then in section \ref{sec-tfc}, we prove that 
\begin{teo}
The polynomial ring $\bZ/p^\delta[f_{1,\delta}, f_{2, \delta}]$ is a $\SL_2(\bZ/p^\delta)$-invariant sub-ring of $\bZ/p^\delta[X, Y]$. Moreover, any invariant
element of order $p^\delta$ (i.e., primitive)
is congruent to some element in $\bZ/p^\delta[f_{1,\delta}, f_{2, \delta}]$ modulo p.
\end{teo}

The polynomial ring $\bZ/p^\delta[f_{1,\delta}, f_{2, \delta}]$ is defined to be a primitive sub-ring
of the invariant sub-ring $\bZ/p^\delta[X, Y]^{\SL_2(\bZ/p^\delta)}$.

From this one can determine the $p$-power
torsion classes in $H^1(X, \tilde{\cM}_n)$.

As for $H^2_c(X, \tilde{\cM}_n)$, a well known result says that $H^2_c(X, \tilde{\cM}_n)\simeq \cM_n/I_{\Gamma}\cM_n$, where $I_{\Gamma}$ is the kernel
of the augmentation map
\[
\bZ[\Gamma]\rightarrow \bZ.
\]

In section \ref{sec-tsc}, we analyze more generally the 
module 
\[
\cM/I_{\SL_2(\bZ)}\cM, \quad \cM=\oplus_{n=0}^{\infty} \cM_n.
\]
In particular, we determine the module structure of $\cM/I_{\SL_2(\bZ)}\cM\otimes \bF_p$ over $(\cM\otimes \bF_p)^{\SL_2(\bF_p)}$,
which is the following theorem( see Proposition \ref{prop-divided-power-struture-modp})

\begin{teo}
We have
\begin{align*}
\cM/I_{\SL_2(\bZ)}\cM\otimes \bF_p&\simeq (\cM\otimes \bF_p)^{\SL_2(\bF_p)} X^{p^2-p}Y^{p-1}\\
&\oplus\bigoplus_{k=2}^{p-1} (\cM\otimes \bF_p)^{\SL_2(\bF_p)}X^{(k-1)(p-1)}Y^{p-1}
\oplus (\cM\otimes \bF_p)^{\SL_2(\bF_p)}1
\end{align*}
where  
\begin{itemize}
\item[(1)] the module $(\cM\otimes \bF_p)^{\SL_2(\bF_p)} X^{p^2-p}Y^{p-1}$ is free of rank one over $(\cM\otimes \bF_p)^{\SL_2(\bF_p)}$;
\item[(2)]the module $(\cM\otimes \bF_p)^{\SL_2(\bF_p)}X^{(k-1)(p-1)}Y^{p-1}$ and $(\cM\otimes \bF_p)^{\SL_2(\bF_p)}1$ are free of rank one over $(\cM\otimes \bF_p)^{\SL_2(\bF_p)}/(f_{1, 1})$.
\end{itemize}
\end{teo}

We furthermore determines the free (primitive, as defined in the paper) part of the module $\cM/I_{\SL_2(\bZ)}\cM\otimes \bZ/p^{\delta}$( which is denoted by $M^{\delta}$ in the paper), which is
the following (see Proposition \ref{prop-second-cohomology-module-structures})

\begin{teo}
The element $$X^{p^{\delta+1}-p^{\delta}+p^{\delta-1}-1}Y^{p^{\delta}-1}$$
under the action of the polynomial ring
$\bZ/p^\delta[f_{1,\delta}, f_{2, \delta}]$, 
generates a sub-module of $\cM/I_{\SL_2(\bZ)}\cM\otimes \bZ/p^{\delta}$, which is free over $\bZ/p^\delta[f_{1,\delta}, f_{2, \delta}]$. Moreover, any element of $\cM/I_{\SL_2(\bZ)}\cM\otimes \bZ/p^{\delta}$ which is of
order $p^{\delta}$ can be written of the form 
\[
c f+h, \quad c\in (\bZ/p^{\delta})^{\times}, \quad p^{\delta-1}h=0
\]
and 
$$f\in \bZ/p^\delta[f_{1,\delta}, f_{2, \delta}]X^{p^{\delta+1}-p^{\delta}+p^{\delta-1}-1}Y^{p^{\delta}-1}.$$

\end{teo}

These two results allow us to determine all the 
p-power torsion elements in $\cM/I_{\SL_2(\bZ)}\cM$.

Finally, in section \ref{sec-cmf}, applying the fundamental
exact sequence from section 1, i.e., the 
exact sequence (\ref{eq1}), we get various 
congruences between cuspidal forms of level 1 and 
Eisenstein series. This recovers the famous 
congruences for Ramanujan $\tau$-function modulo
small primes.

\par \vskip 1pc
{\bf Acknowledgement.}
First of all, I would like to thank Professor G{\"u}nter Harder, for suggesting the problem to me and also for his encouragement and guidance. I thank Robin Bartlett, Bingxiao Liu, Sheng Meng, Carlo Pagano and Danylo Radchenko for many helpful discussions on the article. The whole
article is written during my stay at the Max Planck Institute for Mathematics of Bonn as a postdoc. 
I would like to thank their hospitality. 

\section{Cohomology of Arithmetic Groups}
We follow the notation of \cite{Har}.
Let $\Gamma=\SL_2(\bZ)/\{\pm \Id\}$ and $\mathbb{H}=\{z\in \bC: \Im(z)>0\}$. 
Let 
\[
X=\Gamma\backslash \bH
\]
be the quotient space. We are interested in the Cohomology groups of $X$. 

\begin{definition}
Let 
\[
\cM_n=\{\sum a_v X^{v}Y^{n-v}: a_v\in \bZ, 
\quad 0\leq v\leq n\}
\]
be the space of homogeneous polynomials of degree $n$. We define 
also an action of $\Gamma$ on $\cM_n$, for $g=\begin{pmatrix} 
a & b \\
c & d
\end{pmatrix}$ and $P(X, Y)\in \cM_n $, 
\[
g.P(X, Y)=P(aX+cY, bX+dY). 
\]
This action defines a sheaf on 
$X$, which we denote by $\Tilde{\cM}_n$.
\end{definition}
\remk For more information about 
the sheaf $\tilde{\cM}_n$, we refer
to \cite{Har}.

To study the cohomology of the sheaf $\Tilde{\cM_n}$, we fix some 
generators of the group $\Gamma$
\[
R=\begin{pmatrix} 
1 & -1 \\
1 & 0
\end{pmatrix}, \quad S=\begin{pmatrix} 
0 & -1 \\
1 & 0
\end{pmatrix}, \quad T=RS=\begin{pmatrix} 
1 & 1 \\
0 & 1
\end{pmatrix}
\]

Note that in this special case we have the following.

\begin{prop}
We have 
\[
H^1(X, \Tilde{\cM}_n)\simeq \cM_n/(\cM_n^{<R>}+\cM_n^{<S>})
\]
and
\[
H^1(\partial X, \Tilde{\cM}_n)\simeq \cM_n/((\Id-T)\cM_n)
\]
where $\cM_n^{<R>}$(resp. $\cM_n^{<S>}$) is the sub-module fixed by
$R$(resp. $S$), and $\overline{X}=X\cup \partial X$ is the Borel-Serre 
compactification of $X$. Moreover, we have the following fundamental exact sequence
\begin{align}\label{eqfun}
0&\rightarrow H^0(X, \Tilde{\cM}_n)\rightarrow H^0(\partial X,\Tilde{\cM}_n)\rightarrow
 H^1_c(X, \Tilde{\cM}_n)\rightarrow H^1(X, \Tilde{\cM}_n)\nonumber\\ &\rightarrow H^1(\partial X, \Tilde{\cM}_n) \rightarrow H^2_c(X, \Tilde{\cM}_n)\rightarrow 0.
\end{align}

\end{prop}

\begin{definition}
We define 
\[
H^{1}(X, \Tilde{\cM}_n)_{\int}=\im(H^{1}(X, \Tilde{\cM_n})\rightarrow H^{1}(X, \Tilde{\cM_n}\otimes \bQ)),
\]
\[
H^{1}(X, \Tilde{\cM}_n)_{\tor}=\ker(H^{1}(X, \Tilde{\cM_n})\rightarrow H^{1}(X, \Tilde{\cM_n})_{\int}),
\]
and 
\[
H^{1}_{!}(X, \Tilde{\cM}_n)=\im(H^{1}_c(X, \Tilde{\cM_n})\rightarrow H^{1}(X, \Tilde{\cM_n})),
\]
\[
H^{1}_{!}(X, \Tilde{\cM}_n)_{\int}=\im(H^{1}_{!}(X, \Tilde{\cM}_n)\rightarrow H^{1}(X, \Tilde{\cM}_n)_{\int}).
\]
Similarly, we define $H^{1}(\partial X, \Tilde{\cM}_n)_{\int}, H^{1}(\partial X, \Tilde{\cM}_n)_{\tor}$.
\end{definition}

Now we have the following commutative diagram of exact sequences

\begin{displaymath}
\xymatrix{
 & & 0\ar[d]& 0\ar[d] &\\
 & & H^1(X, \Tilde{\cM}_n)_{\tor}\ar[r]\ar[d]& H^1(\partial X, \Tilde{\cM}_n)_{\tor}\ar[d]&\\
0\ar[r]&H^1_{!}(X, \Tilde{\cM}_n)\ar[d]\ar[r]&H^1(X, \Tilde{\cM}_n)\ar[d]\ar[r]& H^1(\partial X, \Tilde{\cM}_n)\ar[d]&\\
0\ar[r]&H^1(X, \Tilde{\cM}_n)_{\int, !}\ar[r]& H^1(X, \Tilde{\cM}_n)_{\int}\ar[d]\ar[r]&H^1(\partial X, \Tilde{\cM}_n)_{\int}\ar[d]\ar[r]&0\\
& & 0 & 0
}
\end{displaymath}

Applying the Snake Lemma to last two exact sequences in columns, we get 
\begin{align}
0&\rightarrow H^1(X, \Tilde{\cM}_n)_{\int, !}/H^1_{!}(X, \Tilde{\cM}_n)_{\int}\rightarrow
H^1(\partial X, \Tilde{\cM}_n)_{\tor}/H^1(X, \Tilde{\cM}_n)_{\tor}\rightarrow\nonumber \\
&\rightarrow H^2_c(X, \Tilde{\cM}_n)\rightarrow 0\label{eq1}
\end{align}
The object of study in this paper is the fundamental exact sequence (\ref{eq1}).

\remk We should remark that all the terms
are torsions and non-vanishing in general.

\section{Torsions in the Cohomology of boundary}

In this section we study the torsions in the first cohomology of the boundary $\partial X$. 
We show the semi-simplicity of the Hecke action on them and compute the Hecke eigenvalues.

\begin{definition}
We introduce a new set of elements in $\cM_n$
\[
\epsilon_0^{n}=X^n, \epsilon_k^n=Y(Y-X)\cdots (Y-kX+X)X^{n-k},\quad 1\leq k\leq n.
\]
When there is no confusion about degrees, we use $\epsilon_k$ instead. Also, we 
will never take the product of $\epsilon_i^n$ and $\epsilon_j^n$.
\end{definition}

\remk
In the literature, the element 
$(X)_k=X(X-1)\cdots (X-k+1)$ is called a Pochhammer symbol or falling factorial.

\begin{prop}\label{prop1}
Let $n>0$. The set $\{\epsilon_k:0\leq k\leq n\}$ form a basis for $\cM_n$, i.e, 
\[
\cM_n=\oplus_{k=0}^{n}\bZ\epsilon_k.
\]
Moreover, 
\[
T\epsilon_k=\epsilon_k+k\epsilon_{k-1}, 
\]
therefore, we have
\[
\cM_n/(\Id-T)\cM_n=\bZ Y^n\bigoplus \oplus_{k=1}^{n}(\bZ/k\bZ) \epsilon_{k-1}.
\]
\end{prop}

\begin{proof}
We have 
\[
\epsilon_k=\sum_{j=0}^k s(k, j)X^{n-j}Y^{j}
\]
where $(-1)^{k-j}s(k, j)$ is Stirling number  of the first kind. Conversely, we have
\[
X^{n-k}Y^k=\sum_{j=0}^{k}\stirlingii{k}{j}\epsilon_j
\]
where $\stirlingii{k}{j}$ is Stirling number  of the second kind. Therefore 
the set $\{\epsilon_k: 0\leq k\leq n\}$ forms a basis for $\cM_n$. 
\end{proof}

We are ready to compute the Hecke action on the boundary cohomology. Following 
Harder(cf. \cite{Har} \S 3.3), we know that the Hecke operator $T_p$ acts on $H^1(\partial X, \Tilde{\cM}_n)$ as follows
\[
T_p(X^{n-k}Y^k)=p^k\sum_{j=0}^{p-1}X^{n-k}(Y+jX)^k+ p^{n-k}X^{n-k}Y^{k}
\]

Therefore

\begin{prop}\label{prop-Hecke-operator-boudary-eigenvalue}
Let $p>n$ be prime. The Hecke operator $T_p$ acts semi-simply on $H^1(\partial X, \Tilde{\cM}_n)$ with 
\[
T_p(\epsilon_k)=(p^{n-k}+p^{k+1})\epsilon_k, \quad 1\leq k\leq n.
\]
\end{prop}

\remk Note that our proposition applies equally to the free 
part of the cohomology with generator $\epsilon_n$.

\begin{proof}
We prove the statement by induction on $k$. For $k=1$, our statement is trivial. For $k>1$, since
\[
X^{n-k}Y^k=\sum_{j=0}^{k}\stirlingii{k}{j}\epsilon_{j}
\]
and $\stirlingii{k}{k}=1$, by induction we have 
\begin{align*}
T_p(\epsilon_k)&=T_p(X^{n-k}Y^k)-\sum_{j=0}^{k-1} \stirlingii{k}{j}\epsilon_{j}
\\
               &=(p^{n-k}+p^{k+1})\sum_{j=0}^k \stirlingii{k}{j}\epsilon_{j}-\sum_{j=0}^{k-1}(p^{n-j}+p^{j+1})\stirlingii{k}{j}\epsilon_{j}\\
               &+\sum_{j=0}^{k-1}\sum_{i=1}^{k-j}(i+1)p^k(p-i)\frac{(j+i)!}{j!}\stirlingii{k}{j+i}\epsilon_j\\
               \\
               &=(p^{n-k}+p^{k+1})\epsilon_k\\
               &+\sum_{j=0}^{k-1}((p^{n-k}+p^{k+1}-p^{n-j}-p^{j+1})\stirlingii{k}{j}+\sum_{i=1}^{k-j}(i+1)p^k(p-i)\frac{(j+i)!}{j!}\stirlingii{k}{j+i})\epsilon_{j}.
\end{align*}
Note that here we use the fact that 
\[
X^k(Y+iX)^{n-k}=T^i(X^kY^{n-k})=\sum_{j=0}^{k}\stirlingii{k}{j}T^i(\epsilon_{j})
\]
and
\[
T^i(\epsilon_j)=\epsilon_j+2j\epsilon_{j-1}+3j(j-1)\epsilon_{j-2}+\cdots.
\]

We need to show that 
\[
(p^{n-k}+p^{k+1}-p^{n-j}-p^{j+1})\stirlingii{k}{j}+\sum_{i=1}^{k-j}(i+1)p^k(p-i)\frac{(j+i)!}{j!}\stirlingii{k}{j+i}\equiv 0 \text{ mod }j+1.
\]
We observe that for $i>0$, 
\[
\frac{(j+i)!}{j!}\equiv 0, \quad \mod j+1, 
\]
hence, we only need to show 
\[
(p^{n-k}+p^{k+1}-p^{n-j}-p^{j+1})\stirlingii{k}{j}\equiv 0, \quad \mod j+1
\]

Note that this also holds for $j=k$ for trivial reasons. To show this we need the following lemma

\begin{lemma}(cf.\cite{Sta97} Page 57)\label{lemstir}
We have the following identity
\[
\sum_{r=k}^{\infty}\stirlingii{r}{k} X^{r-k}=\frac{1}{(1-X)(1-2X)\cdots(1-kX)}.
\]
\end{lemma}

Consider the following polynomial
\[
P(X, Y)=\sum_{k=j}^{\infty}\sum_{n=k}^{\infty}\stirlingii{k}{j}(p^{n-k}+p^{k+1}-p^{n-j}-p^{j+1})X^{n-k}Y^{k-j}.
\]
We have 
\begin{align*}
P(X, Y)&=\sum_{k=j}^{\infty}(\stirlingii{k}{j}\frac{Y^{k-j}}{1-pX}+\stirlingii{k}{j}\frac{p^{k+1}Y^{k-j}}{1-X}-\stirlingii{k}{j}\frac{p^{k-j}Y^{k-j}}{1-pX}-\stirlingii{k}{j}\frac{p^{j+1}Y^{k-j}}{1-X})\\
&=\frac{1}{(1-pX)}\frac{1}{(1-Y)\cdots (1-j Y)}+\frac{p^{j+1}}{(1-X)}\frac{1}{(1-pY)\cdots (1-pj Y)}\\
&-\frac{1}{(1-pX)}\frac{1}{(1-pY)\cdots (1-pj Y)}-\frac{p^{j+1}}{(1-X)}\frac{1}{(1-Y)\cdots (1-j Y)}\\
&=(\frac{1}{1-pX}-\frac{p^{j+1}}{1-X})(\frac{1}{(1-Y)\cdots (1-j Y)}-\frac{1}{(1-pY)\cdots (1-pj Y)}).
\end{align*}
Note that the condtion that $p>n\geq k>j$ implies $(p, j+1)=1$, hence 
\[
\frac{1}{(1-Y)\cdots (1-j Y)}\equiv \frac{1}{(1-pY)\cdots (1-pj Y)} \text{ mod }j+1
\]
which in turn implies that 
\[
P(X, Y)\equiv 0 \text{ mod }j+1.
\]
This finishes the proof.
\end{proof}

\remk 
Our proposition might fail for $p<n$, consider for example $j=17, k=23, n=24, p=3$, then 
\[
(p^{n-k}+p^{k+1}-p^{n-j}-p^{j+1})\stirlingii{k}{j}\equiv 6 \quad\mod 18
\]
However, a weaker statement holds without the assumption on $p$, i.e, 
\[
T_p(\epsilon_k)\equiv (p^{n-k}+p^{k+1})\epsilon_k \quad \mod q
\]
for any prime $q|(k+1)$. In fact, we have
\[
(p^{n-k}+p^{k+1}-p^{n-j}-p^{j+1})\stirlingii{k}{j}\equiv 0 \quad\mod p.
\]
for any $p$ and $n>k$. Combining this and the argument in the proposition implies our weaker assertion.

Before we finish this section, we deduce from lemma \ref{lemstir} some congruence properties of Stirling number of the second kind which will be used in the next section.

\begin{cor}\label{corstir}
Let $1\leq t, k\leq p$. We have 
\[
\stirlingii{p^2-tp}{kp-1}\equiv \left\{\begin{array}{lcr}
1 \mod p,&  \text{ if } k=1 \text{ and }t=1, \\
0 \mod p,&  \text{ otherwise }, 
\end{array}\right. 
\]
and
\[
\stirlingii{t(p-1)}{kp-1}\equiv \left\{\begin{array}{lcr}
1 \mod p,&  \text{ if } k=1, \\
0 \mod p,&  \text{ otherwise }.
\end{array}\right. 
\]
Finally, we have
\[
\stirlingii{p^2-1}{kp-1}\equiv \left\{\begin{array}{lcr}
1 \mod p,&  \text{ if } k=1 \text{ or } k=p, \\
0 \mod p,&  \text{ otherwise }.
\end{array}\right. 
\]
\end{cor}

\begin{proof}
By lemma \ref{lemstir}, we have 
\[
\sum_{r=kp-1}^{\infty}\stirlingii{r}{kp-1} X^{r-kp+1}=\frac{1}{(1-X)(1-2X)\cdots(1-(kp-1)X)}.
\]
But 
\[
\frac{1}{(1-X)(1-2X)\cdots(1-(kp-1)X)}=\frac{1}{(1-X^{p-1})^k}\equiv \sum_{\ell=1}^{\infty} a_{\ell} X^{\ell(p-1)}\text{ mod } p.
\]
We can assume that $p^2-tp\geq kp-1$, which imply $t+k\leq p$.
For $1\leq k\leq p-1, 1\leq r\leq p-1$, then $(p-1)\mid p(p-t)-kp+1$ would imply $t+k=2$ or $p+1$.
Hence we must have $t+k=2$, i.e, $t=k=1$.
By comparing the coefficients, we get the result. As for 
\[
\stirlingii{t(p-1)}{kp-1}
\]
we observe that $(p-1)\mid t(p-1)-kp+1$ only if $k=1$. Hence it follows that
\[
\stirlingii{t(p-1)}{kp-1}\equiv \left\{\begin{array}{lcr}
1 \text{ mod }p,&  \text{ if } k=1, \\
0 \text{ mod }p,&  \text{ otherwise }.
\end{array}\right. 
\]
Same argument shows the case of $\stirlingii{p^2-1}{kp-1}$.
\end{proof}

\section{Torsions in the first cohomology}\label{sec-tfc}

We fix an odd prime $p>3$. We recall the following theorem of E.L.Dickson (for $\SL_2(\bZ)$),

\begin{teo}
The group $\SL_2(\bZ)$ acts on the polynomial ring $\bF_p[X, Y]$ with the ring of invariants 
a polynomial ring generated by 
\[
f_1=X^pY-XY^p, \quad f_2=\frac{X^{p^2-1}-Y^{p^2-1}}{X^{p-1}-Y^{p-1}}.
\]
\end{teo}

Our first goal in this section is to generalize this theorem to allow $p$-power torsions.

\begin{definition}
Let $G$-be a group and $M$ be a $G$-module which is free over $\bZ/p^n$. Let $(i, M_{\prim}^{G})$ be 
the pair where the embedding $i:M^G_{\prim}\rightarrow M^G$ realizes $M^G_{\prim}$
as one of the maximal $\bZ/p^n$-sub-modules of $M^{G}$ which is free (over $\bZ/p^n$). We call it primitive invariant sub-module
of $G$ over $\bZ/p^n$. We also call an element primitive if
it is of order $p^n$ in $M^G$.
\end{definition}
\remk By elementary divisor decomposition theorem, we know the pair $(i, M^G_{\prim})$ 
always exists and is not unique. When no ambiguity arises, we also drop the morphism $i$
ans say simply that $M^G_{\prim}$ is a primitive invariant sub-module.

Then we have the following 

\begin{teo}
Let $p>3$.
The group $\SL_2(\bZ)$ acts on on the polynomial ring $\bZ/p^n[X, Y]$ with a polynomial ring of primitive invariants generated by 
\[
f_{1, n}=(X^pY-XY^p)^{p^{n-1}}, \quad f_{2, n}=(\frac{X^{p^2-1}-Y^{p^2-1}}{X^{p-1}-Y^{p-1}})^{p^{n-1}}.
\]
\end{teo}

\begin{proof}
We prove this theorem by induction on $n$. The $n=1$ case is just the theorem of Dickson. Assume 
$n>1$ from now on. 
Note that since the action of $\SL_2(\bZ)$ factor through $\SL_2(\bZ/p^n)$ we are allowed to 
replace it by the latter. Now let $G_n=\SL_2(\bZ/p^n)$, and 
\[
L_n=\langle \begin{pmatrix}1&1\\ 0& 1\end{pmatrix}\rangle \subseteq G_n.
\]
 Let $M^n=\bZ/p^n[X, Y]$, then we have the following 
morphisms of cohomology groups
\[
Res: H^1(G_n, M^n)\rightarrow H^1(L_n, M^n),
\]
\[
Inf: H^1(L_n, M^n)\rightarrow H^1(L_{\infty}, M^n),
\]
which are the restriction and inflation, here $L_{\infty}=\langle T\rangle \subseteq \SL_2(\bZ)$.
Note that we have the following exact sequence
\begin{displaymath}
\xymatrix{
0\ar[r] &M^{1}\ar[r]& M^n\ar[r]& M^{n-1}\ar[r] &0
}
\end{displaymath}
which induces long exact sequence
\begin{displaymath}
\xymatrix{
0 \ar[r]& H^0(G_n, M^{1})\ar[r]& H^0(G_n, M^n)\ar[r]& H^0(G_n, M^{n-1})\ar[r]^{\delta_n}& H^1(G_n, M^{1})
}
\end{displaymath}

We get the  following morphisms 
\[
\delta_n: H^0(G_{n-1}, M^{n-1})\rightarrow H^1(G_{n}, M^{1}), 
\]
\[
r_{n}=Inf\circ Res\circ \delta_{n}: H^0(G_{n-1}, M^{n-1})\rightarrow H^1(L_{\infty}, M^{1}).
\]
The morphism $\delta_n$ admits the following description: let $h\in G_{n}$, then for any $f\in H^0(G_{n-1}, M^{n-1})$, and 
$\tilde{f}\in M^n$ be a lift, then 
\[
\delta_{n}(f)(h)=\frac{h(\tilde{f})-\tilde{f}}{p^{n-1}}, 
\]
here we identify the element $ \varphi\in H^1(G_{n}, M^{1})$ as 
\[
\varphi:G_{n}\rightarrow M^{1}, \quad \varphi(h_1h_2)=\varphi(h_1)+h_1\varphi(h_2).
\]
\begin{lemma}\label{lem1}
The morphisms $\delta_{n}$ and $r_n$ are additive and for $f, g\in H^0(G_{n-1}, M^{n-1})$, 
\[
\delta_{n}(fg)=g\delta_{n}(f)+f\delta_{n}(g),\quad r_{n}(fg)=g r_{n}(f)+fr_{n}(g).
\]
\end{lemma}
\begin{proof}[Proof of Lemma \ref{lem1}]
We only prove this lemma for $\delta_n$( it is similar for $r_n$). In fact, for $h\in G_{n}$, and $\tilde{f}, \tilde{g}\in M^n$ be lifting, 
\[
\delta_{n}(fg)(h)=\frac{h(\tilde{f}\tilde{g})-\tilde{f}\tilde{g}}{p^{n-1}}=h(\tilde{f})\delta_{n}(g)(h)+\tilde{g}\delta_{n}(f)(h), 
\]
by assumption, we know that $f$ mod $p^{n-1}$ lands in $H^0(G_{n-1}, M^{n-1})$, therefore we know that
\[
h(\tilde{f})\equiv f \text{ mod } p^{n-1}.
\]
The additivity is obvious. Hence we finish the proof of the lemma.
\end{proof}

As an outcome, we know that 
\[
\delta_{n}(f^p)=0,\quad \forall f\in H^0(G_{n-1}, M^{n-1}).
\]
By induction, we know that $H^0(G_n, M^{n-1})=H^0(G_{n-1}, M^{n-1})$ contains a primitive polynomial ring 
generated by 
\[
f_{1, n-1}=(X^pY-XY^p)^{p^{n-2}}, \quad f_{2, n-1}=(\frac{X^{p^2-1}-Y^{p^2-1}}{X^{p-1}-Y^{p-1}})^{p^{n-2}}.
\]
We pick a lift of $f_{1, n-1}$ and $f_{2,n-1}$ to $M_n$
\[
\tilde{f}_{1, n-1}=(X^pY-XY^p)^{p^{n-2}}, \quad \tilde{f}_{2, n-1}=(\frac{X^{p^2-1}-Y^{p^2-1}}{X^{p-1}-Y^{p-1}})^{p^{n-2}}.
\]
We have the following 

\begin{lemma}\label{lem2}
Let $p>3$. Let $f\in  H^0(G_{n-1}, M^{n-1})$ be a polynomial such that $f=f_{1, n-1}^af_{2, n-1}^b$ with $p\nmid (a, b)$, 
here $(a, b)$ denotes the gcd of $a$ and $b$. Then 
we have $r_n(f)$ is non-trivial in $H^1(L_{\infty}, M^{1})$. More generally, 
let $f=\sum_i c_i f_{1, n-1}^{a_i}f_{2, n-1}^{b_i}$ such that $p\nmid (a_i, b_i)$, then 
$r_n(f)=0$ implies $c_i=0$ in $\bF_p$.
\end{lemma}
We postpone the proof of this lemma to the end of the section.
Assuming this lemma, we still need to show that 
$f_{1, n}$ and $f_{2, n}$ lie in $H^0(G_n,M^{n})$(Note that both of them are of order $p^n$). For $p>2$, the invariance under $S$ is obvious. 
As for the action of $R$, for $p>3$, 
\begin{align*}
    R(f_{1, n})&=((X+Y)^p(-X)-(X+Y)(-X)^p)^{p^{n-1}}\\
               &=(YX^p-XY^p+p(\cdots))^{p^{n-1}}\\
               &=(X^pY-XY^p)^{p^{n-1}}+p^n(\cdots)
\end{align*}
the second term vanishes in $M^n$ by applying lemma \ref{lem-binom-valuation}.
And
\begin{align*}
    R(f_{2, n})&=(\frac{(X+Y)^{p^2-1}-(-X)^{p^2-1}}{(X+Y)^{p-1}-(-X)^{p-1}})^{p^{n-1}}\\
               &=((X+Y)^{p(p-1)}+(X+Y)^{(p-1)(p-1)}X^{(p-1)}+\cdots+X^{p(p-1)})^{p^{n-1}}.
\end{align*}
To apply lemma \ref{lem-binom-valuation}, it remains to see that 
\begin{align*}
&(X+Y)^{p(p-1)}+(X+Y)^{(p-1)(p-1)}X^{(p-1)}+\cdots+X^{p(p-1)}\\
&\equiv Y^{p(p-1)}+Y^{(p-1)(p-1)}X^{(p-1)}+\cdots+X^{p(p-1)}\text{ mod } p.
\end{align*}
But this follows from the fact that $f_{2, 1}$ is invariant under $\SL_{2}(\bF_p)$.
Finally, we need to show the algebraic independence of $f_{1, n}$ and $f_{2, n}$. 
This follows from the fact that their images under the canonical projection into
$M^1$ are algebraically independent, since they are $p^{n-1}$-powers of the 
algebraically independent elements $f_{1, 1}$
and $f_{1, 2}$. We are done.
\end{proof}

With the above theorem, one can proceed to compute the torsions in 
$H^1(X, \Tilde{\cM}_{n})$. 

\begin{lemma}(cf. \cite{Har} \S 2.1)\label{lemma-Harder-exact-functor}
Let $A$ be a ring such that $2, 3$ are inverted. Then functor from the category of 
$\SL_2(\bZ)$-modules with coefficients in $A$ to the category of abelian sheaves on $X$ with coefficients in $A$  is exact. 
\end{lemma}

Therefore we have the following short exact sequence of sheaves
\begin{displaymath}
\xymatrix{
0\ar[r]&\Tilde{\cM}_n\ar[r]^{p^{\delta}} &\Tilde{\cM}_n\ar[r]&\Tilde{\cM}_n\otimes \bZ/p^{\delta}\ar[r]&0,
}
\end{displaymath}
which induces a long exact sequence
\begin{displaymath}
\xymatrix{
0\ar[r]&H^0(X, \Tilde{\cM}_n)\ar[r]&H^0(X, \Tilde{\cM}_n)\ar[r]& H^0(X, \Tilde{\cM}_n\otimes \bZ/p^{\delta})\ar[r]^{\hspace{1.5cm}\alpha}&\\
\ar[r]^{\hspace{-0.8cm}\alpha}& H^1(X, \Tilde{\cM}_n)\ar[r]^{p^{\delta}}&H^1(X, \Tilde{\cM}_n)\ar[r]&H^1(X, \Tilde{\cM}_n\otimes \bZ/p^{\delta})&
}
\end{displaymath}

\begin{cor}\label{torh1}
Let $p>3$. Assume that $n>0$. We have an isomorphism
\[
\alpha: H^0(X, \Tilde{\cM}_n\otimes \bZ/p^{\delta})_{\prim}\rightarrow H^1(X, \Tilde{\cM}_n)[p^{\delta}]_{\prim}, 
\]
where the latter denotes the primitive $p^{\delta}$-torsions in $ H^1(X, \Tilde{\cM}_n)$
which is induced through the morphism $\alpha$.
Moreover, we have
\[
H^0(X, \Tilde{\cM}_n\otimes \bZ/p^{\delta})=(M_{n}^{\delta})^{\Gamma}, 
\]
where $M^{\delta}=\bZ/p^{\delta}[X, Y]$.
\end{cor}

\begin{proof}
We know that for $n>0$, 
\[
H^0(X, \Tilde{\cM}_n)=0, 
\]
this proves the injectivity of $\alpha$. On the other hand, any primitive $p^{\delta}$-torsion
is killed by multiplication by $p^{\delta}$, hence must come from $H^0(X, \Tilde{\cM}_n\otimes \bZ/p^{\delta})$ in the long exact sequence above.
\end{proof}
We finish this section by supplying a proof of Lemma \ref{lem2}.

\begin{proof}[Proof of Lemma \ref{lem2}]
Let $M^1=\oplus_{d=0}^{\infty} M^1_d$, where $M^1_d$ is the subspace of homogeneous polynomials of degree $d$. First of all, we know from Proposition \ref{prop1} that 
\[
H^1(L_{\infty}, M^{1}_{d})=\bZ/p\bZ\epsilon_{d}^d\oplus \bigoplus_{1\leq k\leq d, p\mid k}\bZ/p\bZ \epsilon_{k-1}^d.
\]
Note that here we use the superscript to distinguish the generators for different degrees.
We first compute $r_n(f_{1, n-1})$ and $r_n(f_{2, n-1})$. For $p>3$, 
\begin{align}
r_n(f_{1, n-1})&=\frac{T(\tilde{f}_{1, n-1})-\tilde{f}_{1, n-1}}{p^{n-1}}\nonumber\\    
               &=\frac{(X^p(X+Y)-X(X+Y)^p)^{p^{n-2}}-(X^pY-XY^p)^{p^{n-2}}}{p^{n-1}}\nonumber\\
               &=\frac{(X^p(X+Y)-X(X^p+Y^p+p(XY^{p-1}+X^2g_1)))^{p^{n-2}}-(X^pY-XY^p)^{p^{n-2}}}{p^{n-1}}\nonumber\\
                &=\frac{(X^pY-XY^p-p(X^2Y^{p-1}+X^3g_1))^{p^{n-2}}-(X^pY-XY^p)^{p^{n-2}}}{p^{n-1}}\nonumber\\
                 &=\frac{(X^pY-XY^p)^{p^{n-2}}-p^{n-1}(X^pY-XY^p)^{p^{n-2}-1}(X^2Y^{p-1}+X^3g_1)}{p^{n-1}}\nonumber\\
                 &-\frac{(X^pY-XY^p)^{p^{n-2}}}{p^{n-1}}\label{*}\\
                 &=-(X^pY-XY^p)^{p^{n-2}-1}(X^2Y^{p-1}+X^3g_1)\nonumber\\
                 &=(-1)^{p^{n-2}}X^{p^{n-2}+1}Y^{p^{n-1}-1}+X^{p^{n-2}+2}h_1\label{**}
\end{align}
We make some remarks concerning the computation. Here $h_1, g_1\in \bZ[X, Y]$, and in the expansion (\ref{*}), we 
ignore the terms divisible by $p^{n}$ since by Lemma \ref{lem-binom-valuation}, we have for $2\leq k\leq p^{n-2}$, 
\[
\val_p(p^k\binom{p^{n-2}}{k})\geq k+n-2-\val_p(k)
\]
This guarantees that for $p$ odd,  we have 
\[
\val_p(p^k\binom{p^{n-2}}{k})\geq n.
\]

Therefore, we have
\[
r_n(f_{1, n-1})=(-1)^{p^{n-2}}\epsilon_{p^{n-1}-1}^{p^{n-2}(p+1)}+\sum_{k<p^{n-1}-1}a_k \epsilon_{k}^{p^{n-2}(p+1)}, 
\]
which is nontrivial in $H^1(L_{\infty}, M^{1})$. Similarly, we have 
\begin{align}
r_n(f_{2, n-1})&=\frac{T(\tilde{f}_{2, n-1})-\tilde{f}_{2, n-1}}{p^{n-1}}\nonumber\\
              &=\frac{(X^{p(p-1)}+X^{(p-1)(p-1)}(X+Y)^{p-1}+\cdots +(X+Y)^{p(p-1)})^{p^{n-2}}}{p^{n-1}}\nonumber\\
              &-\frac{(X^{p(p-1)}+X^{(p-1)(p-1)}Y^{p-1}+\cdots +Y^{p(p-1)})^{p^{n-2}} }{p^{n-1}}\nonumber\\
              &=\frac{(X^{p(p-1)}+X^{(p-1)(p-1)}Y^{p-1}+\cdots+Y^{p(p-1)}+p(p-1)XY^{p(p-1)}+X^2g_2)^{p^{n-2}}}{p^{n-1}}\nonumber\\
              &-\frac{(X^{p(p-1)}+X^{(p-1)(p-1)}Y^{p-1}+\cdots +Y^{p(p-1)})^{p^{n-2}} }{p^{n-1}}\nonumber\\
               &=\frac{(X^{p(p-1)}+X^{(p-1)(p-1)}Y^{p-1}+\cdots+Y^{p(p-1)})^{p^{n-2}}}{p^{n-1}}\nonumber\\
               &+\frac{p^{n-1}(p-1)(X^{p(p-1)}+X^{(p-1)(p-1)}Y^{p-1}+\cdots+Y^{p(p-1)})^{p^{n-2}-1}XY^{p(p-1)}}{p^{n-1}}\nonumber\\
              &-\frac{(X^{p(p-1)}+X^{(p-1)(p-1)}Y^{p-1}+\cdots +Y^{p(p-1)})^{p^{n-2}} }{p^{n-1}}+X^2h_2\label{***}\\
              &=(p-1)XY^{p^{n-1}(p-1)-1}+X^2h_3\nonumber
\end{align}
Here again $h_2,h_3, g_2\in \bZ[X, Y]$, and in the expansion (\ref{***}), we 
ignore the terms divisible by $p^{n}$.
\[
r_n(f_{2, n-1})=(p-1)\epsilon_{p^{n-1}(p-1)-1}^{{p^{n-1}(p-1)}}+\sum_{k<p^{n-1}(p-1)-1}b_k \epsilon_{k}^{{p^{n-1}(p-1)}}, 
\]
which is nontrivial in $H^1(L_{\infty}, M^{1})$. 
We will also need to compute the image of 
$f_{1, n-1}f_{2, n-1}=(X^{p^2}Y-XY^{p^2})^{p^{n-2}}$ under $r_n$, which is
\begin{align}
r_n(f_{1, n-1}f_{2, n-1})&=\frac{T(\tilde{f}_{1, n-1}\tilde{f}_{2, n-1})-\tilde{f}_{1, n-1}\tilde{f}_{2, n-1}}{p^{n-1}}\nonumber\\
   &=\frac{(X^{p^2}(X+Y)-X(X+Y)^{p^2})^{p^{n-2}}-(X^{p^2}Y-XY^{p^2})^{p^{n-2}}}{p^{n-1}}\nonumber\\
               &=\frac{(X^{p^2}(X+Y)-X(X^{p^2}+Y^{p^2}+\sum_{i=1}^{p-1}\binom{p^2}{ip}X^{ip}Y^{p^2-ip}+p^2g_3))^{p^{n-2}}}{p^{n-1}}\nonumber\\
               &-\frac{(X^{p^2}Y-XY^{p^2})^{p^{n-2}}}{p^{n-1}}\nonumber\\
                &=\frac{(X^{p^2}Y-XY^{p^2}-(\sum_{i=1}^{p-1}\binom{p^2}{ip}X^{ip+1}Y^{p^2-ip}+p^2Xg_3))^{p^{n-2}}}{p^{n-1}}\nonumber\\
                &-\frac{(X^{p^2}Y-XY^{p^2})^{p^{n-2}}}{p^{n-1}}\nonumber\\
                 &=\frac{(X^{p^2}Y-XY^{p^2})^{p^{n-2}}-(X^{p^2}Y-XY^{p^2})^{p^{n-2}}}{p^{n-1}}\nonumber\\
                 &-\frac{p^{n-1}(X^{p^2}Y-XY^{p^2})^{p^{n-2}-1}(\sum_{i=1}^{p-1}\frac{1}{p}\binom{p^2}{ip}X^{ip+1}Y^{p^2-ip}+pXg_3)}{p^{n-1}}\nonumber\\
             &=-(X^{p^2}Y-XY^{p^2})^{p^{n-2}-1}(\sum_{i=1}^{p-1}\frac{1}{p-i}\binom{p-1}{i}X^{ip+1}Y^{p^2-ip})\label{****}
\end{align}
Here $g_3\in \bZ[X, Y]$, and in the expansion (\ref{****}), we 
ignore the terms divisible by $p^{n}$ and use the fact that $\frac{1}{p}\binom{p^2}{ip}\equiv \frac{1}{p-i}\binom{p-1}{i}$\text{ mod } $p$. By Corollary \ref{corstir}, we know that the term 
$\sum_{i=2}^{p-1}\frac{1}{p-i}\binom{p-1}{i}X^{ip+1}Y^{p^2-ip}$ vanishes in $H^1(L_{\infty}, M^1)$, which implies that it
lies in the image of $(T-1)$.
Now applying the fact that 
in $M^1$, 
\begin{align}\label{equation-invariant-under-T-p}
(T-1)((X^{p^2}Y-XY^{p^2})^{p^{n-2}-1})=0, 
\end{align}
we know that the term
\[
(X^{p^2}Y-XY^{p^2})^{p^{n-2}-1}(\sum_{i=2}^{p-1}\frac{1}{p-i}\binom{p-1}{i}X^{ip+1}Y^{p^2-ip})
\]
vanishes in $H^1(L_{\infty}, M^1)$. 
Therefore, 
\[
r_n(f_{1, n-1}f_{2, n-1})=(X^{p^2}Y-XY^{p^2})^{p^{n-2}-1}X^{p+1}Y^{p^2-p}+\cdots
\]
Again, Corollary \ref{corstir} tells us that
\[
X^{p+1}Y^{p^2-p}=\epsilon_{p-1}^{p^2+1}=X^{p^2-p+2}Y^{p-1}
\]
in $H^1(L_{\infty}, M^1)$. Hence applying again
equation (\ref{equation-invariant-under-T-p}) allows us to obtain 
\begin{align*}
r_n(f_{1, n-1}f_{2, n-1})&=(X^{p^2}Y-XY^{p^2})^{p^{n-2}-1}X^{p^2-p+2}Y^{p-1}\\
&=(-1)^{p^2-1}\epsilon_{p^n-p^2+p-1}^{p^{n-2}(p^2+1)}+\sum_{k<p^n-p^2+p-1}d_k\epsilon_{k}^{p^{n-2}(p^2+1)}.
\end{align*}

Note that by property of $r_n$, we have 
\[
r_n(f^k)=kf^{k-1}r_n(f).
\]
Therefore, 
\begin{align*}
r_n(f_{1, n-1}^af_{2, n-1}^b)&=af_{1, n-1}^{a-1}f_{2, n-1}^br_n(f_{1, n-1})+bf_{1, n-1}^{a}f_{2, n-1}^{b-1}r_n(f_{2, n-1})\\
                             &=(-1)^{p^{n-2}a}(a-b)\epsilon_{p^{n-1}(a+(p-1)b)-1}^{p^{n-2}(a(p+1)+p(p-1)b)}+\cdots 
\end{align*}
So if $a-b\neq 0$\text{ mod } $p$, we know that $r_n(f_{1, n-1}^af_{2, n-1}^b)$ is non-zero
in $H^1(L_{\infty}, M^1)$. Assume that $p\mid(a-b)$ but $p\nmid (a, b)$. We show that 
$r_n(f_{1, n-1}^af_{2, n-1}^b)$ does not vanish in $H^1(L_{\infty}, M^1)$. We argue under the assumption
\[
a=b+ps, s\geq 0, 
\]
which is similar for the case $b\geq a$. In fact, we have
\begin{align*}
r_n(f_{1, n-1}^{ps}(f_{1, n-1}f_{2, n-1})^b)&=bf_{1, n-1}^{ps}(f_{1, n-1}f_{2, n-1})^{b-1}r_n(f_{1, n-1}f_{2, n-1})\\
&=(-1)^{(a-1)p^{n-2}}b\epsilon_{p^{n-1}(a+(p-1)b)-p^2+p-1}^{p^{n-2}(a(p+1)+p(p-1)b)}+\cdots
\end{align*}
This shows the non-vanishing of $r_n(f_{1, n-1}^af_{2, n-1}^b)$.
Finally, if
\[
\sum_{i=1}^{\ell}r_n(c_i f_{1,n-1}^{a_i}f_{2, n-1}^{b_i})=0
\]
such that
\begin{align}\label{equation-degree-equality}
 d=p^{n-2}(a_i(p+1)+p(p-1)b_i), \ell=1, \cdots, \ell.
\end{align}
Assume first $n>2$, then 
\[
p^{n-1}(a+(p-1)b)-1\not\equiv  p^{n-1}(a+(p-1)b)-p^2+p-1 \mod p^{n-1}.
\]
And for $n=2$, the equality
\[
p^{n-1}(a+(p-1)b)-1\neq p^{n-1}(a+(p-1)b)-p^2+p-1
\]
imply
\[
a_i+(p-1)(b_i-1)=a_j+(p-1)b_j.
\]
But from (\ref{equation-degree-equality}), we get
\[
a_i+(p-1)b_i=a_j+(p-1)b_j.
\]
We deduce from it that 
\[
p-1=0, 
\]
which is absurd.
Therefore we are reduced to following two cases:

(1) We have $a_i-b_i\not\equiv  0\mod p$
for all $i$, but the equations
\begin{align*}
p^{n-1}(a_i+(p-1)b_i)-1&=p^{n-1}(a_j+(p-1)b_j)-1\\
p^{n-2}(a_i(p+1)+p(p-1)b_i)&=p^{n-2}(a_j(p+1)+p(p-1)b_j)
\end{align*}
imply $a_i=a_j, b_i=b_j$. Hence we must have
\[
r_n(c_if_{1,n-1}^{a_i}f_{2, n-1}^{b_i})=0,  i=1, \cdots, \ell.
\]
This shows $c_i\equiv 0\mod p$.

(2)We have $a_i-b_i\equiv 0\mod p$
and $p\nmid b_i$ for all $i$, then 
\begin{align*}
p^{n-1}(a_i+(p-1)b_i)-p^2+p-1&=p^{n-1}(a_j+(p-1)b_j)-p^2+p-1\\
p^{n-2}(a_i(p+1)+p(p-1)b_i)&=p^{n-2}(a_j(p+1)+p(p-1)b_j)
\end{align*}
imply also $a_i=a_j, b_i=b_j$, from which we deduce
that $c_i\equiv 0 \mod p$.

\end{proof}

\section{Torsions in Second Cohomology with Compact Support}\label{sec-tsc}

In this section, we determine the torsions appearingappear in $H^2_c(X, \Tilde{\cM}_n)$.
As in the previous section, we fix a prime
$p>3$.

\begin{definition}
Let $I_{\Gamma}$ be the augmentation ideal of the group algebra
\[
\nu: \bZ[\Gamma]\rightarrow \bZ, \quad \sum_{i} a_i g_i\mapsto \sum_{i} a_i.
\]
\end{definition}

\begin{prop}(cf. \cite{Har11Vol1}, \S 4.8.5)
We have 
\[
H^2_c(X, \Tilde{\cM}_n)=\cM_n/I_{\Gamma}\cM_n.
\]
\end{prop}

Therefore we need to compute the coinvariants of 
$\cM_n \otimes \bZ/p^{\delta}$ under the natural action of 
$\SL_2(\bZ)$. 

We follow the strategy of \cite{Lew17}, where the authors treat the 
case of $\GL_2(\bF_p^r)$ acting on $\bF_{p^r}[X, Y]$. We remark that though the strategy is the
same, their method does not yield the case $\SL_2(\bF_p)$ due to the lack of construction of
some auxiliary linear functions. Instead, our study
of invariants  in the divided power rings gives
naturally such linear functions. 

\begin{definition}
Let $G$ be a group acting on a module $M$. Then we let 
\[
M_G=M/I_GM
\]
be the space of coinvariants of $G$. In case $G_{\delta}=SL_2(\bZ/p^{\delta})$ and $M^{\delta}=\bZ/p^{\delta}[X, Y]$, let
\[
\Hilb(M_{G_{\delta}}^{\delta}, t)=\sum_{d\geq 0}\rank_{\bZ/p^{\delta}}(M^{\delta}_{G_{\delta}})_{d}t^d
\]
be the Hilbert series of $M^{\delta}_{G_{\delta}}$, where $(M^{\delta}_{G_{\delta}})_d$ be the degree $d$ part of $M^{\delta}_{G_{\delta}}$.
\end{definition}

\remk Although the module $M^{\delta}_{G_{\delta}}$ is not free over
the ring $\bZ/p^{\delta}$, by elementary 
divisor theorem it still makes sense to
speak about the rank of the free part(in the decomposition).

Before we state and prove the main result, we recall some preliminary results on divided power rings, for details, see \cite{Ak82}.
\begin{def-prop}
Set $V_{\delta}=(\bZ/p^{\delta})^2$. Then regarded as a Hopf algebra, the algebra $M^{\delta}=\bZ/p^{\delta}[X, Y]=\Sym(V_{\delta}^{*})$ admits a (restricted) dual Hopf algebra
\[
D(V_{\delta})=\bZ/p^{\delta}[\xi_1, \xi_2],
\]
where $D(V_{\delta})_d=(M^{\delta}_{d})^*$, with $\xi_1$ dual to $X$ and
$\xi_2$ dual to $Y$. Moreover,  $D(V_{\delta})$ carries a divided power structure satisfying
\[
\xi_i^{(m)}\xi_i^{(n)}=\binom{m+n}{n}\xi_i^{(m+n)}, \quad \text{ for } i=1, 2.
\]
\end{def-prop}

\begin{prop}
The divided power ring $D(V_{\delta})$ admits an action of $G_{\delta}$
by
\[
\begin{pmatrix} 
a & b \\
c & d
\end{pmatrix}.f(\xi_1, \xi_2)=f(a\xi_1+b\xi_2, c\xi_1+d\xi_2)
\]
satisfying
\[
\langle gf, h\rangle=\langle f, gh\rangle, \quad \forall f\in D(V_{\delta})_d, h\in M^{\delta}_d, g\in G^{\delta}, 
\]
where $\langle , \rangle: D(V)_d\times M^{\delta}_d\rightarrow \bZ/p^{\delta}$ being 
the natural pairing.
\end{prop}

\begin{proof}
We only need to check that we have
\[
\langle gf, h\rangle=\langle f, gh\rangle
\]
for $f$ (resp. $g$) running through the basis $\{\xi_1^{(m)}\xi_2^{(n)}: m+n=d\}$ (resp. $\{X^{m}Y^{n}: m+n=d\}$). We have
\begin{align*}
\langle S(\xi_1^{(m)}\xi_2^{(n)}), X^rY^s\rangle &=\langle (-1)^m\xi_2^{(m)}\xi_1^{(n)}, X^rY^s\rangle\\
&=(-1)^m\langle\xi_1^{(n)}, X^r\rangle \langle\xi_2^{(m)}, Y^s\rangle\\
&=(-1)^m\delta_{n, r}\delta_{m,s}, 
\end{align*}
and similarly, 
\begin{align*}
\langle \xi_1^{(m)}\xi_2^{(n)}, S(X^rY^s)\rangle&=\langle\xi_1^{(m)}\xi_2^{(n)}, Y^r(-X)^s\rangle\\
&=(-1)^s\langle\xi_1^{(m)}, X^s\rangle\langle\xi_2^{(n)}, Y^r\rangle\\
&=(-1)^m\delta_{n, r}\delta_{m,s}.
\end{align*}
Therefore
\[
\langle S(\xi_1^{(m)}\xi_2^{(n)}), X^rY^s\rangle=\langle\xi_1^{(m)}\xi_2^{(n)}, S(X^rY^s)\rangle.
\]
Also
\begin{align*}
\langle T(\xi_1^{(m)}\xi_2^{(n)}), X^rY^s\rangle&=\langle (\xi_1+\xi_2)^{(m)}\xi_2^{(n)}, X^rY^s\rangle\\
&=\langle\sum_{k=0}^m\xi_1^{(m-k)}\xi_2^{(k)}\xi_2^{(n)}, X^rY^s\rangle\\
&=\sum_{k=0}^{m}\binom{n+k}{k}\langle\xi_1^{(m-k)}\xi_2^{(n+k)}, X^rY^s\rangle\\
&=\sum_{k=0}^{m}\binom{n+k}{k}\delta_{m-k, r}\delta_{n+k, s}\\
&=\binom{n+m-r}{m-r}
\end{align*}
and 
\begin{align*}
\langle\xi_1^{(m)}\xi_2^{(n)}, T(X^rY^s)\rangle&=\langle\xi_1^{(m)}\xi_2^{(n)}, X^r(X+Y)^s\rangle\\
&=\langle\xi_1^{(m)}\xi_2^{(n)},\sum_{k=0}^s \binom{s}{k}X^{r+k}Y^{s-k}\rangle\\
&=\sum_{k=0}^{s}\binom{s}{k}\langle\xi_1^{(m)}\xi_2^{(n)}, X^{r+k}Y^{s-k}\rangle\\
&=\sum_{k=0}^{s}\binom{s}{k}\delta_{m, r+k}\delta_{n, s-k}\\
&=\binom{s}{m-r}\delta_{m-r, s-n}\\
&=\binom{m+n-r}{m-r}
\end{align*}
therefore
\[
\langle T(\xi_1^{(m)}\xi_2^{(n)}), X^rY^s\rangle=\langle \xi_1^{(m)}\xi_2^{(n)}, T(X^rY^s)\rangle.
\]
\end{proof}

\begin{cor}
The pairing $\langle , \rangle$ induces a morphism 
\[
\varphi_{\delta}:D(V_{\delta})^{G_{\delta}}\rightarrow (M^{\delta}_{G_{\delta}})^{*}
\]
which induces an isomorphism 
\[
\varphi_{1}:D(V_{1})^{G_{1}}\rightarrow (M^{1}_{G_{1}})^{*}
\]
\end{cor}
\begin{proof}
The morphism $\varphi_1$ is an isomorphism due to the fact that $\bZ/p$ is
a field.
\end{proof}
\remk 
Although we do not give a proof, the reader should be aware that we have an isomorphism
\[
\varphi_{\delta}:D(V_{\delta})^{G_{\delta}}_{\prim}\rightarrow (M^{\delta}_{G_{\delta}, \prim})^{*}, 
\]
of course, the module $M^{\delta}_{G_{\delta}, \prim}$ should be appropriately defined.

\begin{prop}\label{divel}
We have a set of elements belonging to $D(V_{1})^{G_1}$
\[
\{\sum_{k=1}^{n-1}\xi_1^{(k(p-1))}\xi_2^{((n-k)(p-1))}:n\geq 2\}.
\]
\end{prop}
\begin{proof}
Indeed, 
\begin{align*}
&T(\sum_{k=1}^{n-1}\xi_1^{(k(p-1))}\xi_2^{((n-k)(p-1))})\\
&= \sum_{k=1}^{n-1}(\xi_1+\xi_2)^{(k(p-1))}\xi_2^{((n-k)(p-1))}\\    
&=\sum_{k=1}^{n-1}(\sum_{r=0}^{k(p-1)}\xi_1^{(r)}\xi_2^{(k(p-1)-r)})\xi_2^{((n-k)(p-1))}\\
&=\sum_{k=1}^{n-1}\sum_{r=0}^{k(p-1)}\binom{n(p-1)-r}{(n-k)(p-1)}\xi_1^{(r)}\xi_2^{(n(p-1)-r)}\\
&=\sum_{r=0}^{(n-1)(p-1)-1}\sum_{\frac{r+1}{p-1}\leq k\leq n-1}\binom{n(p-1)-r}{(n-k)(p-1)}\xi_1^{(r)}\xi_2^{(n(p-1)-r)}\\
&+\sum_{k=1}^{n-1}\xi_1^{(k(p-1))}\xi_2^{((n-k)(p-1))}.
\end{align*}
Note that it is enough to show 
\[
\sum_{\frac{r+1}{p-1}\leq k\leq n-1}\binom{n(p-1)-r}{(n-k)(p-1)}\equiv 0 \quad \mod p
\]
or, equivalently,  
\[
\sum_{q-1\mid k,1\leq  k\leq j-1}\binom{j}{k}\equiv 0 \quad \mod p, \forall j>0.
\]
Now we want to use the following trick 
\[
\sum_{\beta\in \bF_p^{\times}}\beta^k=\left\{\begin{array}{lcr}
p-1=-1, &\text{ if } k=\ell(p-1) \text{ for some } \ell\in \bZ,  \\
0, &\text{ otherwise.}
\end{array}\right.
\]
Let 
\[
h(t)=(1+t)^j-1-t^j=\sum_{k=1}^{j-1}\binom{j}{k}t^k.
\]
Then we have 
\begin{align*}
\sum_{q-1\mid k,1\leq  k\leq j-1}\binom{j}{k}&=-\sum_{\beta\in \bF_p^{\times}} h(\beta)\\
&=-\sum_{\beta\in \bF_p} h(\beta)\\
&=-\sum_{\beta\in \bF_p}(\beta+1)^j+\sum_{\beta\in \bF_p}\beta^j+\sum_{\beta\in \bF_p}1\\
&=0.
\end{align*}
\end{proof}

We first study the module structure of $M^1_{G_1}$.
\begin{teo}
We have
\[
\Hilb(M^1_{G_1}, t)=1+\frac{t^{2(p-1)}}{1-t^{p-1}}+\frac{t^{p(p+1)}}{(1-t^{p+1})(1-t^{p(p-1)})}.
\]
\end{teo}

\remk 
One easily rewrites the expression in the theorem as follows
\[
\Hilb(M^1_{G_1}, t)=\frac{1+t^{2(p-1)}+t^{3(p-1)}+\cdots +t^{(p-1)^2}}{1-t^{p(p-1)}}+\frac{t^{p^2-1}}{(1-t^{p+1})(1-t^{p(p-1)})}
\]
\remk The analogue theorem holds with $\bF_{p}$ replaced by any $\bF_{p^r}$. Also, 
similar proof can be produced for $\SL_n(\bF_{p^r})(n\geq 3)$(under the condition that we have a good understanding of the boundary cohomology of certain locally symmetric space). But since we are 
only interested in the $n=2$ case for present, we leave the case $n\geq 3$ for future work.

Note that our theorem is a consequence of the proposition below.

\begin{prop}\label{prop-divided-power-struture-modp}
We have the following structure decomposition of $M_{G_1}^1$, 
\[
M_{G_1}^1\simeq M^{1, G_1}\epsilon_{p-1}^{p^2-1}\oplus\bigoplus_{k=2}^{p-1} M^{1, G_1}\epsilon_{p-1}^{k(p-1)}\oplus M^{1, G_1}1
\]
where 
\begin{itemize}
\item[(1)] the module $M^{1, G_1}\epsilon_{p-1}^{p^2-1}$ is free of rank one over $M^{1, G_1}$;
\item[(2)]the module $M^{1, G_1}\epsilon_{p-1}^{k(p-1)}(2\leq k\leq p-1)$ and $M^{1, G_1}1$ are free of rank one over $M^{1, G_1}/(f_1)$.
\end{itemize}
\end{prop}

\remk We remark that in $M^{1, G_1}$, 
\[
\epsilon_{p-1}^{k(p-1)}=X^{(k-1)(p-1)}Y^{p-1} (k=2, \cdots p-1), \quad \epsilon_{p-1}^{p^2-1}=X^{p^2-p}Y^{p-1}.
\]

\begin{proof}
We start with the canonical subjective morphism 
\[
\pi: M^1/(1-T)M^1\rightarrow M^{1, G_1}.
\]
First of all, we know from proposition \ref{prop1} that 
\[
M^1_d/(1-T)M^1_d=\bF_p\epsilon_{d}^d\oplus \bigoplus_{1\leq \ell\leq d, p\mid \ell}\bF_p \epsilon_{\ell-1}^d.
\]
Consider the case $d=p-1$, then $\epsilon_{p-1}^{p-1}=Y^{p-1}$. The fact that 
$\pi(\epsilon_{p-1}^{p-1})=0$ follows from 
\[
\pi(Y^{p-1})=X^{p-1}+(S-Id)(X^{p-1})
\]
and $X^{p-1}=\epsilon_{0}^{p-1}=0$ in $M^1/(1-T)M^1$.
We use Proposition \ref{divel} to show that $\epsilon_{p-1}^{k(p-1)}(2\leq k\leq p-1)$ 
does not vanish in $M_{G_1}^1$. In fact, 
by Corollary \ref{corstir}, we have
\[
\epsilon_{p-1}^{k(p-1)}=X^{(k-1)(p-1)}Y^{p-1} 
\]
in $M^1_{k(p-1)}/(1-T)M^1_{k(p-1)}$. Since
\[
\langle \sum_{r=1}^{k-1}\xi_1^{(r(p-1))}\xi_2^{((k-r)(p-1))}, X^{(k-1)(p-1)}Y^{p-1}\rangle=1
\]
we know that $\pi(\epsilon_{p-1}^{k(p-1)})\neq 0$ in $M^1_{G_1}$.

\begin{lemma}\label{nonvb}
The elements $\{\epsilon_{p-1}^{k(p-1)}: 2\leq k\leq p-1\}\cup \{1\}$ are all annihilated by 
$f_1$ but not annihilated by any power of $f_2$.
\end{lemma}

\begin{proof}[Proof of Lemma \ref{nonvb}]We have  
\[
f_1=X^pY-XY^p=(\Id-S)(X^pY).
\]
For $2\leq k\leq p-1$, consider
\[
h_k=\sum_{i=1}^{k-1}X^{i(p-1)}Y^{(k-i)(p-1)}.
\]
Note that 
\begin{align*}
f_1h_k&=(X^pY-XY^p)(\sum_{i=0}^{k}X^{i(p-1)}Y^{(k-i)(p-1)}-X^{k(p-1)}-Y^{k(p-1)})\\
 &=X^{(k+1)(p-1)+1}Y-XY^{(k+1)(p-1)+1}\\
 &+X^{k(p-1)+p}Y+X^pY^{k(p-1)+1}-X^{k(p-1)+1}Y^p-XY^{k(p-1)+p}\\
 &=(\Id-S)(X^{(k+1)(p-1)+1}Y+X^{k(p-1)+p}Y+X^pY^{k(p-1)+1})
\end{align*}
It remains to see that 
\[
h_k=(k-1)\epsilon_{p-1}^{k(p-1)}\neq 0
\]
in $M^1_{G_1}$. In fact, we have
\[
X^{i(p-1)}Y^{(k-i)(p-1)}=\sum_{r=0}^{(k-i)(p-1)}\stirlingii{(k-i)(p-1)}{r}\epsilon^{k(p-1)}_{r}
\]
only the terms with $r=-1 \mod p$ remains in $M^1/(1-T)M^1$. But applying Corollary
\ref{corstir}, we know that for $1\leq i\leq k-1$, 
\[
X^{i(p-1)}Y^{(k-i)(p-1)}=\epsilon^{k(p-1)}_{p-1}.
\]
Hence we have proved that $f_1$ annihilated all the elements in $\{\epsilon_{p-1}^{k(p-1)}: 2\leq k\leq p-1\}\cup \{1\}$. We still need to show that any power of $f_2$ does not annihilate any element in the 
same set.
We know that 
\[
f_2=X^{p(p-1)}+X^{(p-1)(p-1)}Y^{p-1}+\cdots+Y^{p(p-1)}.
\]
And 
\[
f_2^j=Y^{jp(p-1)}+\sum_{m, n\geq 1}c_{m, n}X^{m(p-1)}Y^{n(p-1)}+X^{jp(p-1)}.
\]
And we note that $\sum_{m, n\geq 1}c_{m, n}=(p+1)^j-2$. Therefore, 
\[
\langle\sum_{r=1}^{jp-1}\xi_1^{(r(p-1))}\xi_2^{((jp-r)(p-1))}, f_2^j\rangle=\sum_{m, n\geq 1}c_{m, n}=(p+1)^j-2\equiv-1 \mod p.
\]
And for $i\geq 1$, 
\[
\langle \sum_{r=1}^{jp+i}\xi_1^{(r(p-1))}\xi_2^{((jp+i+1-r)(p-1))}, f_2^jX^{i(p-1)}Y^{p-1}\rangle=(1+p)^{j}\equiv 1 \mod p.
\]
We finish the proof of the lemma.
\end{proof}

\begin{lemma}\label{vashmon}
Let $d>0$. Then the monomial $X^{d}Y^{p-1}$ vanishes $M^1_{G_1}$ if $(p-1)\nmid d$. 
\end{lemma}

\begin{proof}[Proof of Lemma \ref{vashmon}]
In fact, let $g=\begin{pmatrix}c&0\\ 0& c^{-1}\end{pmatrix}\in G_1$ with $c\in \bF^{\times}_p$.
Then 
\[
(\Id-g)(X^dY^{p-1})=(1-c^d)X^dY^{p-1}.
\]
If $(p-1)\nmid d$, then picking $c$ with $c^d\neq 1$ shows the result.  
\end{proof}

\begin{lemma}\label{bbase}
The set of elements
\[
\{1, \epsilon_{p-1}^{2(p-1)}, \epsilon_{p-1}^{3(p-1)}, \cdots, \epsilon_{p-1}^{(p-1)^2}\}\cup \{\epsilon_{p-1}^{(p^2-1)}\}
\]
generate $M^1_{G_1}/(f_1, f_2)M^1_{G_1}$ as a vector space over $\bF_p$ and hence generate $M^1_{G_1}$ as module
over $M^{1,G_1}$.
\end{lemma}

\begin{proof}[Proof of Lemma \ref{bbase}]Let $h\in M^1_d=\bF_p[X, Y]_d$.
Then by Euclidean division 
with respect to $Y$, we can write
\[
h=f_2h_1+h_2, \quad h_2=Y^jX^{d-j}+\sum_{\ell<j}c_{\ell} Y^{\ell}X^{d-\ell}, \quad j<p^2-p.
\]
Furthermore,  we have
\[
h_2=f_1h_3+aY^d+bX^{d-p+1}Y^{p-1}+h_4, \quad h_4=\sum_{\ell<p-1}m_{\ell}X^{d-\ell}Y^{\ell}
\]
Therefore in $M^{1}_{G_1}/(f_1, f_2)M^1_{G_1}$, we have
\[
h=h_2=aY^d+bX^{d-p+1}Y^{p-1}+h_4. 
\]
But the term $h_4$ vanishes in $M^1/(1-T)M^1$, we have
\[
h=aY^d+bX^{d-p+1}Y^{p-1}. 
\]
We also know that in $M^1_{G_1}$, 
\[
Y^d=X^d+(S-\Id)(X^d),
\]
while $X^d$ vanishes in $M^1/(1-T)M^1$. Hence
\[
h=bX^{d-p+1}Y^{p-1}
\]
in $M^{1}_{G_1}/(f_1, f_2)M^1_{G_1}$. According to the lemma \ref{vashmon}, this term can only be nonzero when 
\[
(p-1)\mid d.
\]
Assume that $d=(p-1)d_1$.
At this point we invoke the following 
relation between $f_1$ and $f_2$, 
\[
X^{p^2-1}=X^{p-1}f_2-f_1^{p-1}.
\]
Therefore, if $j\geq p^2-1$, then
\[
X^jY^{p-1}=(X^{p-1}f_2-f_1^{p-1})X^{j-p^2+1}Y^{p-1}, 
\]
which vanishes in $M^{1}_{G_1}/(f_1, f_2)M^1_{G_1}$.
Therefore, we can assume $d-p+1< p^2-1$, hence
\[
d_1<p+2, 
\]
then $d=d_1(p-1)\leq p^2-1$. Therefore, we know that the set 
\[
\{1, \epsilon_{p-1}^{2(p-1)}, \cdots, \epsilon_{p-1}^{p(p-1)}, \epsilon_{p-1}^{p^2-1}\}
\]
generates the space $M^{1}_{G_1}/(f_1, f_2)M^1_{G_1}$.
We claim that the element $\epsilon_{p-1}^{p(p-1)}$ also vanishes. 
To show this, consider
\[
f_2=X^{p(p-1)}+(X^{(p-1)(p-1)}Y^{p-1}+\cdots+X^{(p-1)}Y^{(p-1)(p-1)})+Y^{p(p-1)}, 
\]
by Corollary \ref{corstir}, we know that all the terms inside the parenthesis
$$X^{i(p-1)}Y^{(p-i)(p-1)}, \quad 1\leq i\leq p-1$$ are equal to $\epsilon_{p-1}^{p(p-1)}$, which implies 
\[
0=(p-1)\epsilon_{p-1}^{p(p-1)}=-\epsilon_{p-1}^{p(p-1)}.
\]
The second assertion in the lemma follows from the first via the following lemma
\begin{lemma}(cf. \cite{Lew17} Proposition B.14)
Let $R$ be an $\bN$-graded ring. Let $I\subset R_{+}:=\oplus_{d>0}R_d$ be 
a homogeneous ideal of positive degree elements. Let $M$ be a $\bZ$-graded 
$R$-module with nonzero degrees bounded below. 
Then a subset generates $M$ as $R$-module if and only if its images generate
$M/IM$ as $R/I$-module.
\end{lemma}
\end{proof}

Now we can finish the proof of Proposition \ref{prop-divided-power-struture-modp}. Let
$N_1$ and $N_2$ be the $M^{1, G_1}$ sub-modules of $M_{G_1}^1$ generated by
$\epsilon_{p-1}^{p^2-1}$ and $\{\epsilon_{p-1}^{k(p-1)}: 2\leq k\leq  p-1\}\cup\{1\}$. Then lemma \ref{bbase} implies
\[
M^{1, G_1}=N_1+N_2.
\]
And lemma \ref{nonvb} shows that 
\[
N_2=\oplus_{k=2}^{p-1}M^{1, G_1}/(f_1)\epsilon_{p-1}^{k(p-1)}\oplus M^{1, G_1}/(f_1) 1.
\]
We claim that $N_1\simeq M^{1, G_1}$. In fact, if $f\in M^{1, G_1}$ annihilates
$\epsilon_{p-1}^{p^2-1}$. Then $ff_1$ annihilates the whole module
$M^{1, G_1}$, which contradicts the following

\begin{prop}(cf. \cite{Lew17} Proposition 5.7)
Any finite group $G$ of automorphisms of an integral domain $S$ has
$\rank_{S^G}(S_G)=1$. 
\end{prop}
Finally, we conclude that the sum $N_1+N_2$ is direct since
\[
N_1\cap N_2\subset \Ann_{M^{1, G_1}}(f_1)\cap N_1=0.
\]
\end{proof}

We still need to consider the case of $M^{\delta}_{G_\delta}$ for $\delta>1$.
We have the following

\begin{teo}\label{thm-second-cohomology-torions-}
We have
\[
\Hilb(M_{G_{\delta}}^{\delta}, t)=1+\frac{t^{p^{\delta+1}+p^{\delta-1}-2}}{(1-t^{p^{\delta-1}(p+1)})(1-t^{p^{\delta}(p-1)})}.
\]
\end{teo}
Again, this theorem is a consequence of the following

\begin{prop}\label{prop-second-cohomology-module-structures}
Assume $p$>3. Let $M^{2, G_2}_{\prim}=\bZ/p^{\delta}[f_{1, \delta}, f_{2, \delta}]$. Then the sub-module of $M^{\delta}_{G_{\delta}}$
generated by the element $X^{p^{\delta+1}-p^{\delta}+p^{\delta-1}-1}Y^{p^{\delta}-1}$ over
$M^{2, G_2}_{\prim}$ is free of rank one. Moreover, 
the direct sum of this module and a copy of  $\bZ/p^{\delta}$
generated by the degree zero element $1$, which is denoted by
$M^{\delta}_{G_{\delta}}$, forms a primitive
sub-module of $M^{\delta}_{G_{\delta}}$.
\end{prop}
\remk For $\delta=1$, it is covered by previous case.

As before, we need some results
on the divided power rings.

\begin{prop}\label{dbtor}
Let $V_{\infty}=\bZ^2$. Then 
we have 
\[
D(V_{\infty})_d/(\Id-T)=\bZ \nu_0\oplus \bigoplus_{i=1}^{d}\bZ/i\bZ \nu_{i}
\]
with
\[
\nu_{i}=\sum_{j=i}^d\stirlingii{j}{i}\xi_1^{(d-j)}\xi_2^{(j)}.
\]
\end{prop}

\remk Note that here by convention, we have
\[
\stirlingii{0}{0}=1, \stirlingii{n}{0}=0, \text{ for } n>0.
\]

\begin{proof}
We want to show that 
\[
(T-\Id)\nu_i=(i+1)\nu_{i+1}
\]
In fact, 
\begin{align*}
(T-\Id)\nu_i&=\sum_{j=i}^d\stirlingii{j}{i}(\xi_1+\xi_2)^{(d-j)}\xi_2^{(j)}-\nu_i \\  
    &=\sum_{j=i}^d\sum_{k=1}^{d-j}\stirlingii{j}{i}\xi_1^{(d-j-k)}\xi_2^{(k)}\xi_2^{(j)}\\
    &=\sum_{j=i}^d\sum_{k=1}^{d-j}\stirlingii{j}{i}\binom{k+j}{j}\xi_1^{(d-j-k)}\xi_2^{(j+k)}\\
    &=\sum_{h=i+1}^d\sum_{j=i}^{h-1}\stirlingii{j}{i}\binom{h}{j}\xi_1^{(d-h)}\xi_2^{(h)}.
\end{align*}
Therefore, we need to show that
\[
\sum_{j=i}^{h-1}\stirlingii{j}{i}\binom{h}{j}=(i+1)\stirlingii{h}{i+1}.
\]
Note that we have
\[
\sum_{j=i}^{\infty}\stirlingii{j}{i}t^j=\frac{t^i}{(1-t)(1-2t)\cdots (1-it)}.
\]
Then 
\begin{align*}
&\sum_{h=i+1}^{\infty}\sum_{j=i}^{h-1}\stirlingii{j}{i}\binom{h}{j}t^h\\
&=\sum_{j=i}^{\infty}\stirlingii{j}{i}t^j\sum_{h=j+1}^{\infty}\binom{h}{j}t^{h-j}\\
&=\sum_{j=i}^{\infty}\stirlingii{j}{i}t^j(\frac{1}{(1-t)^{j+1}}-1)\\
&=\frac{(\frac{t}{1-t})^i}{(1-t)(1-\frac{t}{1-t})(1-\frac{2t}{1-t})\cdots (1-\frac{it}{1-t})}
-\frac{t^i}{(1-t)(1-2t)\cdots (1-it)}\\
&=\frac{(i+1)t^{i+1}}{(1-t)(1-2t)\cdots (1-(i+1)t)}.
\end{align*}
We are done.
\end{proof}

\begin{cor} \label{corpoder}
The element $\nu_i$ in 
$D(V_{\infty})_d/(\Id-T)$ is of order divisible by $p$ if and only $p\mid i$.
Moreover, the exact $p$-power of $\nu_i$ is $p^{\val_p(i)}$, where $\val_p$
is the standard $p$-adic valuation.
\end{cor}

\begin{prop}
The set of elements
\[
\{\sum_{j=1}^{p}\xi_{1}^{(p^{\delta-1}-1+jp^{\delta-1}(p-1))}\xi_{2}^{(p^{\delta-1}-1 +(p-j+1)p^{\delta-1}(p-1))}: \delta\geq 1\}
\]
belongs to $D(V_1)^{G_1}$.
\end{prop}

\begin{proof}
For $\delta=1$, we are covered by
proposition \ref{divel}. 
By design, the element 
\[
u_{\delta}:=\sum_{j=1}^{p}\xi_{1}^{(p^{\delta-1}-1+jp^{\delta-1}(p-1))}\xi_{2}^{(p^{\delta-1}-1 +(p-j+1)p^{\delta-1}(p-1))}
\]
is symmetric with respect to $\xi_1$ and $\xi_2$.
Therefore we only need to show that it is invariant
under $T$. In fact, 
\begin{align*}
T(u_{\delta})&=\sum_{j=1}^{p}(\xi_{1}+\xi_{2})^{(p^{\delta-1}-1+jp^{\delta-1}(p-1))}\xi_{2}^{(p^{\delta-1}-1 +(p-j+1)p^{\delta-1}(p-1))}\\
&=\sum_{j=1}^{p}\sum_{\ell=0}^{p^{\delta-1}-1+jp^{\delta-1}(p-1)}\xi_{1}^{(\ell)}\xi_2^{(p^{\delta-1}-1+jp^{\delta-1}(p-1)-\ell)}\xi_{2}^{(p^{\delta-1}-1 +(p-j+1)p^{\delta-1}(p-1))}\\
&=\sum_{j=1}^{p}\sum_{\ell=0}^{p^{\delta-1}-2+jp^{\delta-1}(p-1)}\binom{p^{\delta+1}+p^{\delta-1}-2-\ell}{p^{\delta-1}-1 +(p-j+1)p^{\delta-1}(p-1)}\xi_{1}^{(\ell)}\xi_{2}^{(p^{\delta+1}+p^{\delta-1}-2 -\ell)}\\
&+\sum_{j=1}^{p}\xi_{1}^{(p^{\delta-1}-1+jp^{\delta-1}(p-1))}\xi_{2}^{(p^{\delta-1}-1 +(p-j+1)p^{\delta-1}(p-1))}\\
&=\sum_{\ell=0}^{p^{\delta+1}-p^{\delta}+p^{\delta-1}-2}\sum_{\frac{\ell+2-p^{\delta-1}}{p^{\delta-1}(p-1)}\leq j\leq p }\binom{p^{\delta+1}+p^{\delta-1}-2-\ell}{p^{\delta-1}-1 +(p-j+1)p^{\delta-1}(p-1)}\xi_{1}^{(\ell)}\xi_{2}^{(p^{\delta+1}+p^{\delta-1}-2 -\ell)}\\
&+\sum_{j=1}^{p}\xi_{1}^{(p^{\delta-1}-1+jp^{\delta-1}(p-1))}\xi_{2}^{(p^{\delta-1}-1 +(p-j+1)p^{\delta-1}(p-1))}.\\
\end{align*}
As in Proposition \ref{divel}, we need to show that for fixed $j\geq 0$, 
\begin{equation}\label{coneq}
\sum_{k: 1< kp^{\delta-1}(p-1)<j}\binom{p^{\delta-1}-1+j}{p^{\delta-1}-1 +kp^{\delta-1}(p-1)}\equiv 0 \quad \mod p.
\end{equation}
We prove this equality by induction on $\delta$.
The case $\delta=1$ is proved in Proposition \ref{divel}. Assume that $\delta>1$. 
We recall the following congruence property of binomial coefficients.

\begin{lemma}(Lucas's theorem)\label{lemma-Lucas-theorem}
Assume that we have
\[
m=pm_1+m_2, n=pn_1+n_2, \quad 0\leq m_2<p, 0\leq m_2<p
\]
then 
\[
\binom{m}{n}\equiv \binom{m_1}{n_1}\binom{m_2}{n_2}\quad \mod p.
\]
\end{lemma}

Assume now $j=pj_1+j_2$ with $0\leq j_2<p$. If $j_2>0$, then by the 
obove lemma, we have
\[
\binom{p^{\delta-1}-1+j}{p^{\delta-1}-1 +kp^{\delta-1}(p-1)}\equiv \binom{j_2-1}{p-1}\binom{p^{\delta-2}+j_1}{p^{\delta-2}-1+kp^{\delta-2}(p-1)}\mod p
\]
but by assumption $j_2-1<p-1$, therefore we get $\binom{j_2-1}{p-1}=0$, hence
\[
\binom{p^{\delta-1}-1+j}{p^{\delta-1}-1 +kp^{\delta-1}(p-1)}\equiv 0 \quad \mod p.
\]
Now assume $ j=pj_1$, then we have
\[
\binom{p^{\delta-1}-1+j}{p^{\delta-1}-1 +kp^{\delta-1}(p-1)}\equiv \binom{p^{\delta-2}-1+j_1}{p^{\delta-2}-1+kp^{\delta-2}(p-1)}\mod p.
\]
The left hand side of (\ref{coneq}) becomes 
\[
\sum_{k: 1< kp^{\delta-2}(p-1)<j_1}\binom{p^{\delta-2}-1+j_1}{p^{\delta-2}-1 +kp^{\delta-2}(p-1)}, 
\]
applying induction, we know that it vanishes in $\bF_p$.
\end{proof}

Furhtermore,

\begin{prop}\label{prop-lift-fixedelement-divpow}
The element
\[
u_{\delta}:=\sum_{j=1}^{p}\xi_{1}^{(p^{\delta-1}-1+jp^{\delta-1}(p-1))}\xi_{2}^{(p^{\delta-1}-1 +(p-j+1)p^{\delta-1}(p-1))}
\]
lift to a primitive element $w_{\delta}$ in $D(V_{\delta})^{G_{\delta}}$.
\end{prop}

\begin{lemma}\label{lemma-congruence-levelone}
We have
\[
u_{\delta}=\nu_{p^{\delta}-1}-\nu_{p^{\delta+1}-1}
\]
in $D(V_1)$.
\end{lemma}

\begin{proof}[Proof of lemma \ref{lemma-congruence-levelone}]
By lemma \ref{lemstir}, we have
\[
\sum_{n=p^{\delta}-1}^{\infty}\stirlingii{n}{p^{\delta}-1}t^n=\frac{t^{p^{\delta}-1}}{(1-t)(1-2t)\cdots (1-(p^{\delta}-1)t)}.
\]
The right hand side equals to 
\[
\frac{t^{p^{\delta}-1}}{(1-t^{p-1})^{p^{\delta-1}}}=\frac{t^{p^{\delta}-1}}{1-t^{(p-1)p^{\delta-1}}}=\sum_{j=0}^{\infty}t^{p^{\delta}-1+j(p-1)p^{\delta-1}}
\]
in $\bF_p[t]$. 
Similarly, 
\[
\sum_{n=p^{\delta+1}-1}^{\infty}\stirlingii{n}{p^{\delta+1}-1}t^n=\sum_{j=0}^{\infty}t^{p^{\delta+1}-1+j(p-1)p^{\delta}}.
\]
Hence we have
\[
\nu_{p^{\delta}-1}=\sum_{j=0}^{p}\xi_{1}^{(p^{\delta}-1+(p-j-1)p^{\delta-1}(p-1))}\xi_{2}^{(p^{\delta}-1 +jp^{\delta-1}(p-1))}=u_{\delta}+\nu_{p^{\delta+1}-1}.
\]
\end{proof}

\begin{proof}[Proof of Proposition \ref{prop-lift-fixedelement-divpow}]
First of all, we have the following 
short exact sequence of $G_{\delta}$-modules
\[
0\rightarrow D(V_{\delta-1})\rightarrow D(V_{\delta})
\rightarrow D(V_1)\rightarrow 0, 
\]
which induces the following long exact 
sequence 
\begin{displaymath}
\xymatrix{
0\ar[r]&H^0(G_{\delta}, D(V_{\delta-1}))\ar[r]&H^0(G_{\delta}, D(V_{\delta}))\ar[r]&H^0(G_{\delta}, D(V_{1}))\ar[r]^{\hspace{1.5cm}\kappa_{\delta}}&\\
\ar[r]&H^1(G_{\delta}, D(V_{\delta-1}))\ar[r]&H^1(G_{\delta}, D(V_{\delta})).
}
\end{displaymath}
Fix
$t_{\delta}=\nu_{p^{\delta}-1}-\nu_{p^{\delta+1}-1}\in \bZ[\xi_1, \xi_2]$, then by 
Lemma \ref{lemma-congruence-levelone}
we have $t_{\delta}\equiv u_{\delta}\mod p$.
The element $\kappa_{\delta}(u_{\delta})\in H^1(G_{\delta}, D(V_{\delta-1}))$
on the level of cochain is
defined as follows: we have for $g\in G_{\delta}$
\[
\kappa_{\delta}(u_{\delta})(g)=\frac{g(t_{\delta})-t_{\delta}}{p}\in D(V_{\delta-1}), 
\]
this make sence since $g(t_\delta)\equiv t_\delta\mod p$.
We claim that

\begin{lemma}\label{lemma-subgroup-vanishing}
The map 
\[
\kappa_{\delta}(u_{\delta}):G_{\delta}\rightarrow D(V_{\delta-1}), \quad g\mapsto \frac{g(t_{\delta})-t_{\delta}}{p}
\]
vanishes on 
\[
N_{\delta}=\{g\in G_{\delta}:g\equiv \begin{pmatrix}
1&0\\
0&1
\end{pmatrix}\quad \mod p\}.
\]
\end{lemma}
\begin{proof}[Proof of Lemma \ref{lemma-subgroup-vanishing}]
Note that we have an exact sequence of groups
\[
1\rightarrow N_{\delta, \delta-1}\rightarrow N_{\delta}\rightarrow N_{\delta-1}\rightarrow 1
\]
where $N_{\delta, \delta-1}\simeq \bF_p^3$
is generated by 
\[
g_1=\begin{pmatrix}
1+p^{\delta-1}&0\\
0&1-p^{\delta-1}
\end{pmatrix}, 
g_2=\begin{pmatrix}
1& p^{\delta-1}\\
0&1
\end{pmatrix}, 
g_3=\begin{pmatrix}
1&0\\
p^{\delta-1}&1
\end{pmatrix}.
\]
We first show that $\kappa_{\delta}(u_{\delta})$ vanishes on $N_{\delta, \delta-1}$.
By symmetry, we only check that 
\[
\frac{g_i(t_{\delta})-t_{\delta}}{p}=0\in D(V_{\delta-1}), \text{ for }i=1, 2, 3.
\]
Let $d_0=p^{\delta+1}+p^{\delta-1}-2$, then we have 
\begin{align*}
&\frac{g_1(t_{\delta})-t_{\delta}}{p}\\
&=\frac{\sum_{j=p^{\delta}-1}^{d_0}(\stirlingii{j}{p^{\delta}-1}-\stirlingii{j}{p^{\delta+1}-1})((1+p^{\delta-1})\xi_1)^{(d_0-j)}((1-p^{\delta-1})\xi_{2})^{(j)}-t_{\delta}}{p}\\
&=\sum_{j=p^{\delta}-1}^{d_0}\frac{(\stirlingii{j}{p^{\delta}-1}-\stirlingii{j}{p^{\delta+1}-1})((1+p^{\delta-1})^{d_0-j}(1-p^{\delta-1})^{j}-1)}{p}\xi_1^{(d_0-j)}\xi_{2}^{(j)}
\end{align*}
But for $j\geq p^{\delta}-1$ and $\delta\geq 2$,
\begin{align*}
&\frac{(\stirlingii{j}{p^{\delta}-1}-\stirlingii{j}{p^{\delta+1}-1})((1+p^{\delta-1})^{d_0-j}(1-p^{\delta-1})^{j}-1)}{p}\\
&=\frac{(\stirlingii{j}{p^{\delta}-1}-\stirlingii{j}{p^{\delta+1}-1})((1+p^{\delta-1})(1-p^{\delta-1})-1)}{p}\\
&\equiv 0\quad \mod p^{\delta-1}.
\end{align*}
We conclude that 
\[
\frac{g_1(t_{\delta})-t_{\delta}}{p}=0\in D(V_{\delta-1}).
\]
And
\begin{align*}
&\frac{g_2(t_{\delta})-t_{\delta}}{p}\\
&=\frac{\sum_{j=p^{\delta}-1}^{d_0}(\stirlingii{j}{p^{\delta}-1}-\stirlingii{j}{p^{\delta+1}-1)})(\xi_1+p^{\delta-1}\xi_2)^{(d_0-j)}\xi_{2}^{(j)}-t_{\delta}}{p}\\
&\equiv\frac{\sum_{j=p^{\delta}-1}^{d_0}(\stirlingii{j}{p^{\delta}-1}-\stirlingii{j}{p^{\delta+1}-1})p^{\delta-1}\xi_1^{(d_0-j-1)}\xi_2\xi_2^{(j)}}{p}(\mod p^{\delta-1})\\
&= \sum_{j=p^{\delta}-1}^{d_0}(\stirlingii{j}{p^{\delta}-1}-\stirlingii{j}{p^{\delta+1}-1})p^{\delta-2}(j+1)\xi_1^{(d_0-j-1)}\xi_2^{(j+1)}.
\end{align*}
Therefore we need to show 
\[
(\stirlingii{j}{p^{\delta}-1}-\stirlingii{j}{p^{\delta+1}-1})(j+1)\equiv 0, \quad \mod p.
\]
But by Theorem \ref{thm-stirl-valuation}, we know that
\[
\stirlingii{j}{p^{\delta}-1}(j+1)\equiv 0 \quad \mod p^{\delta-1}.
\]
This shows that
\[
\frac{g_2(t_{\delta})-t_{\delta}}{p}=0\in D(V_{\delta-1}).
\]
For $g_3$, 
\begin{align*}
&\frac{g_3(t_{\delta})-t_{\delta}}{p}\\
&=\frac{\sum_{j=p^{\delta}-1}^{d_0}(\stirlingii{j}{p^{\delta}-1}-\stirlingii{j}{p^{\delta+1}-1)})\xi_1^{(d_0-j)}(p^{\delta-1}\xi_1+\xi_{2})^{(j)}-t_{\delta}}{p}\\
&=\frac{\sum_{j=p^{\delta}-1}^{d_0}(\stirlingii{j}{p^{\delta}-1}-\stirlingii{j}{p^{\delta+1}-1})p^{\delta-1}\xi_1^{(d_0-j)}\xi_1\xi_2^{(j-1)}}{p}(\mod p^{\delta-1})\\
&=\sum_{j=p^{\delta}-1}^{d_0}(\stirlingii{j}{p^{\delta}-1}-\stirlingii{j}{p^{\delta+1}-1})p^{\delta-2}(d-j+1)\xi_1^{(d_0-j+1)}\xi_2^{(j-1)}.
\end{align*}
Therefore we need to show 
\[
(\stirlingii{j}{p^{\delta}-1}-\stirlingii{j}{p^{\delta+1}-1})(d-j+1)\equiv 0, \quad \mod p.
\]
But by Theorem \ref{thm-stirl-valuation}, we know that
\[
\val_p(\stirlingii{j}{p^{\delta}-1})\geq \delta-1-\val_p(j+1)
\]
and since 
\[
\val_p(d+2-(j+1))\geq \min\{ \val_p(j+1), \val_p(d+2)=\delta\}, 
\]
we must have
\[
\stirlingii{j}{p^{\delta}-1}(d-j+1)\equiv 0 \quad \mod p^{\delta-1}.
\]
This shows that
\[
\frac{g_3(t_{\delta})-t_{\delta}}{p}=0\in D(V_{\delta-1}).
\]

Therefore $\kappa_{\delta}(u_{\delta})$ vanishes on $N_{\delta, \delta-1}$.
For $g\in N_{\delta}, h\in N_{\delta, \delta-1}$, we have
\[
\kappa_{\delta}(u_{\delta})(gh)=\kappa_{\delta}(u_{\delta})(g)+g\kappa_{\delta}(u_{\delta})(h)=\kappa_{\delta}(u_{\delta})(g)
\]
Moreover, 
\[
\kappa_{\delta}(u_{\delta})(hg)=\kappa_{\delta}(u_{\delta})(h)+h\kappa_{\delta}(u_{\delta})(g)=h\kappa_{\delta}(u_{\delta})(g), 
\]
which shows that $\kappa_{\delta}(u_{\delta})(g)\in D(V_{\delta-1})^{N_{\delta, \delta-1}}$.
Hence $\kappa_{\delta}(u_{\delta})$  defines a co-chain
\[
\kappa_{\delta}(u_{\delta}): N_{\delta-1}\rightarrow D(V_{\delta-1})^{N_{\delta, \delta-1}}.
\]

For $0\leq \gamma\leq \delta-1$, we prove by induction that $\kappa_{\delta}(u_{\delta})$ defines a co-chain map 
\[
\kappa_{\delta}(u_{\delta}):N_{\gamma}\rightarrow D(V_{\delta-1})^{N_{\delta, \gamma}},
\]
where 
\[
1\rightarrow N_{\delta, \gamma}\rightarrow N_{\delta}\rightarrow N_{\gamma}\rightarrow 1.
\]
The case of $\gamma=\delta-1$ is already proved. Assume the co-chain map
\[
\kappa_{\delta}(u_{\delta}):N_{\gamma}\rightarrow D(V_{\delta-1})^{N_{\delta, \gamma}},
\]
we show that it vanishes on $N_{\gamma, \gamma-1}$, where
\[
1\rightarrow N_{\gamma, \gamma-1}\rightarrow N_{\gamma}\rightarrow N_{\gamma-1}\rightarrow 1.
\]
The group $N_{\gamma, \gamma-1}\simeq \bF_p^3$ generated by 
\[
g_4=\begin{pmatrix}
1-p^{\gamma}&0\\
0& (1-p^{\gamma})^{-1}
\end{pmatrix}, 
g_5=\begin{pmatrix}
1& p^{\gamma}\\
0&1
\end{pmatrix}, 
g_6=\begin{pmatrix}
1&0\\
p^{\gamma}&1
\end{pmatrix}, 
\]
here we identify $g_i$ with their lift to $G_{\delta}$ and therefore
\[
(1-p^{\gamma})^{-1}=\sum_{0\leq i\gamma<\delta}p^{i\gamma}.
\]
Again, we check that
\[
\frac{g_i(t_{\delta})-t_{\delta}}{p}=0\in D(V_{\delta-1}), \text{ for }i=4, 5, 6.
\]
We have
\begin{align*}
&\frac{g_4(t_{\delta})-t_{\delta}}{p}\\
&=\frac{\sum_{j=p^{\delta}-1}^{d_0}(\stirlingii{j}{p^{\delta}-1}-\stirlingii{j}{p^{\delta+1}-1})((1-p^{\gamma})\xi_1)^{(d_0-j)}((1-p^{\gamma})^{-1}\xi_{2})^{(j)}-t_{\delta}}{p}\\
&=\sum_{j=p^{\delta}-1}^{d_0}\frac{(\stirlingii{j}{p^{\delta}-1}-\stirlingii{j}{p^{\delta+1}-1})((1-p^{\gamma})^{d_0-2j}-1)}{p}\xi_1^{(d_0-j)}\xi_{2}^{(j)}.
\end{align*}
But 
\begin{align*}
&\val_p((\stirlingii{j}{p^{\delta}-1}-\stirlingii{j}{p^{\delta+1}-1})((1-p^{\gamma})^{d_0-2j}-1))\\
&=\val_p((\stirlingii{j}{p^{\delta}-1}-\stirlingii{j}{p^{\delta+1}-1}))+\val_p((1-p^{\gamma})^{d_0+2}-(1-p^{\gamma})^{2j+2}).\\
&\geq\val_p((\stirlingii{j}{p^{\delta}-1}-\stirlingii{j}{p^{\delta+1}-1}))\\
&+\min\{\val_p((1-p^{\gamma})^{d_0+2}-1),\val_p(1-(1-p^{\gamma})^{2j+2}) \} 
\end{align*}
By Proposition \ref{prop-binomial-expansion-valuation}, we know
\[
\val_p((1-p^{\gamma})^{d_0+2}-1)\geq \val_p(d_0+2)+\gamma=\delta-1+\gamma,
\]
and 
\[
\val_p(1-(1-p^{\gamma})^{2j+2})\geq \val_p(j+1)+\gamma
\]
By Theorem \ref{thm-stirl-valuation},
\[
\val_p(\stirlingii{j}{p^{\delta}-1})\geq \delta-1-\val_p(j+1).
\]
Note that $p^{\delta}-1\leq j\leq d_0$ implies $\val_p(j+1)\leq \delta+1$. 
If $\val_p(j+1)\leq \delta-1$, then 
\[
\min\{\val_p((1-p^{\gamma})^{d_0+2}-1),\val_p(1-(1-p^{\gamma})^{2j+2}) \}\geq \val_p(j+1)+1, 
\]
hence we have
\[
\val_p((\stirlingii{j}{p^{\delta}-1}-\stirlingii{j}{p^{\delta+1}-1})((1-p^{\gamma})^{d_0-2j}-1))\geq \delta.
\]
If $\val_p(j+1)\geq \delta$, we  have
\[
\min\{\val_p((1-p^{\gamma})^{d_0+2}-1),\val_p(1-(1-p^{\gamma})^{2j+2}) \}\geq \delta,
\]
but 
\[
\val_p(\stirlingii{j}{p^{\delta}-1}-\stirlingii{j}{p^{\delta+1}-1})\geq 0, 
\]
hence
\[
\val_p((\stirlingii{j}{p^{\delta}-1}-\stirlingii{j}{p^{\delta+1}-1})((1-p^{\gamma})^{d_0-2j}-1))\geq \delta.
\]
This finishes the proof of 
\[
\frac{g_4(t_{\delta})-t_{\delta}}{p}=0\in D(V_{\delta-1}).
\]
As for $g_5$, 
\begin{align*}
&\frac{g_5(t_{\delta})-t_{\delta}}{p}\\
&=\frac{\sum_{j=p^{\delta}-1}^{d_0}(\stirlingii{j}{p^{\delta}-1}-\stirlingii{j}{p^{\delta+1}-1})(\xi_1+p^{\gamma}\xi_2)^{(d_0-j)}\xi_{2}^{(j)}-t_{\delta}}{p}\\
&=\frac{\sum_{j=p^{\delta}-1}^{d_0}(\stirlingii{j}{p^{\delta}-1}-\stirlingii{j}{p^{\delta+1}-1})\sum_{\ell=1}^{d_0-j}p^{\ell \gamma}\xi_1^{(d_0-j-\ell)}\xi_2^{(\ell)}\xi_2^{(j)}}{p}\\
&=\sum_{j=p^{\delta}-1}^{d_0}\sum_{\ell=1}^{d_0-j}(\stirlingii{j}{p^{\delta}-1}-\stirlingii{j}{p^{\delta+1}-1})p^{\ell\gamma-1}\binom{\ell+j}{j}\xi_1^{(d_0-j-\ell)}\xi_2^{(j+\ell)}\\
&=\sum_{h=p^{\delta}}^{d_0}\sum_{j=p^{\delta}-1}^{h-1}(\stirlingii{j}{p^{\delta}-1}-\stirlingii{j}{p^{\delta+1}-1})p^{(h-j)\gamma-1}\binom{h}{j}\xi_1^{(d_0-h)}\xi_2^{(h)}.
\end{align*}
And consider the following formal series
\begin{align*}
&\sum_{h=p^{\delta}}^{\infty}\sum_{j=p^{\delta}-1}^{h-1}    
\stirlingii{j}{p^{\delta}-1}p^{(h-j)\gamma}\binom{h}{j}t^h\\
&=\sum_{j=p^{\delta}-1}^{\infty}\stirlingii{j}{p^{\delta}-1}p^{-j\gamma}\sum_{h=j+1}^{\infty}\binom{h}{j}(p^\gamma t)^h\\
&=\sum_{j=p^{\delta}-1}^{\infty}\stirlingii{j}{p^{\delta}-1}p^{-j\gamma}(\frac{1}{(1-p^{\gamma}t)^{j+1}}-1)(p^{\gamma}t)^j\\
&=\sum_{j=p^{\delta}-1}^{\infty}\stirlingii{j}{p^{\delta}-1}\frac{t^j}{(1-p^\gamma t)^{j+1}}-\sum_{j=p^{\delta}-1}^{\infty}\stirlingii{j}{p^{\delta}-1}t^j\\
&=\frac{t^{p^{\delta}-1}}{(1-p^\gamma t)(1-(p^\gamma +1)t)\cdots (1-(p^\delta+p^\gamma-1)t)}-\frac{t^{p^{\delta-1}}}{(1-t)(1-2t)\cdots(1-(p^{\delta}-1)t)}\\
&\equiv 0\quad \mod p^{\delta}.
\end{align*}
The same proof applies to also to 
\[
\sum_{h=p^{\delta}}^{\infty}\sum_{j=p^{\delta}-1}^{h-1}    
\stirlingii{j}{p^{\delta+1}-1}p^{(h-j)\gamma}\binom{h}{j}t^h\equiv 0\mod p^{\delta}
\]
This shows that
\[
\frac{g_5(t_{\delta})-t_{\delta}}{p}=0\in D(V_{\delta-1}).
\]
As for $g_6$, 
\begin{align*}
&\frac{g_6(t_{\delta})-t_{\delta}}{p}\\
&=\frac{\sum_{j=p^{\delta}-1}^{d_0}(\stirlingii{j}{p^{\delta}-1}-\stirlingii{j}{p^{\delta+1}-1})\xi_{1}^{(d-j)}(p^{\gamma}\xi_1+\xi_2)^{(j)}-t_{\delta}}{p}\\
&=\frac{\sum_{j=p^{\delta}-1}^{d_0}(\stirlingii{j}{p^{\delta}-1}-\stirlingii{j}{p^{\delta+1}-1})\sum_{\ell=0}^{j-1}p^{(j-\ell) \gamma}\xi_1^{(d_0-j)}\xi_1^{(j-\ell)}\xi_2^{(\ell)}}{p}\\
&=\sum_{j=p^{\delta}-1}^{d_0}\sum_{\ell=0}^{j-1}(\stirlingii{j}{p^{\delta}-1}-\stirlingii{j}{p^{\delta+1}-1})p^{(j-\ell)\gamma-1}\binom{d_0-\ell}{j-\ell}\xi_1^{(d_0-\ell)}\xi_2^{(\ell)}\\
&=\sum_{\ell=0}^{d_0}\sum_{j=\ell+1}^{d_0}(\stirlingii{j}{p^{\delta}-1}-\stirlingii{j}{p^{\delta+1}-1})p^{(j-\ell)\gamma-1}\binom{d_0-\ell}{j-\ell}\xi_1^{(d_0-\ell)}\xi_2^{(\ell)}.
\end{align*}
We put 
\[
H(\ell)=\sum_{j=\ell+1}^{d_0}(\stirlingii{j}{p^{\delta}-1}-\stirlingii{j}{p^{\delta+1}-1})p^{(j-\ell)\gamma-1}\binom{d_0-\ell}{j-\ell}.
\]
We want to show that 
\[
\val_p(H(\ell))\geq \delta-1.
\]
Note that it is enough to show that 
\[
\val_p((\stirlingii{j}{p^{\delta}-1}-\stirlingii{j}{p^{\delta+1}-1})p^{j-\ell-1}\binom{d_0-\ell}{j-\ell})\geq \delta-1.
\]
Indeed, we will show that 
\[
\val_p(\stirlingii{j}{p^{\delta}-1}p^{j-\ell-1}\binom{d_0-\ell}{j-\ell})\geq \delta-1, \text{ for } \ell+1\leq j\leq d_0, 
\]
since the same proof applies to show 
\[
\val_p(\stirlingii{j}{p^{\delta+1}-1}p^{j-\ell-1}\binom{d_0-\ell}{j-\ell})\geq \delta-1.
\]
First of all, note that we have
\begin{align*}
    &\binom{d_0-\ell}{j-\ell}\\
    &=\frac{(d_0-\ell)(d_0-\ell-1)\cdots (d_0-j+1)}{(j-\ell)!}\\
    &=\frac{(d_0-\ell)(d_0-\ell-1)\cdots (d_0-j+2)}{(j-\ell-1)!}\frac{d_0-j+1}{j-\ell}.
\end{align*}
Hence
\[
\val_p(\binom{d_0-\ell}{j-\ell})\geq \val_p(\frac{d_0-j+1}{j-\ell})\geq \val_p(d_0-j+1)-\val_p(j-\ell).
\]
So we get
\begin{align*}
&\val_p(\stirlingii{j}{p^{\delta}-1}p^{j-\ell-1}\binom{d_0-\ell}{j-\ell})\\
&\geq \val_p(\stirlingii{j}{p^{\delta}-1})+j-\ell-1+\val_p(d_0-j+1)-\val_p(j-\ell).
\end{align*}
We know that for $j-\ell\geq 1$, 
\[
j-\ell-1\geq \val_p(j-\ell).
\]
So
\[
\val_p(\stirlingii{j}{p^{\delta}-1}p^{j-\ell-1}\binom{d_0-\ell}{j-\ell})\geq \stirlingii{j}{p^{\delta}-1}+\val_p(d_0-j+1).
\]
If $\val_p(j+1)\geq \delta-1$, then 
\begin{align*}
\val_p(\stirlingii{j}{p^{\delta}-1}p^{j-\ell-1}\binom{d_0-\ell}{j-\ell})&\geq \val_p(d_0-j+1)\\
&\geq \min\{\val_p(d_0+2), \val_p(j+1)\}\\
&\geq \delta-1.
\end{align*}
And if $\val_p(j+1)<\delta-1$, then $\val_p(d_0-j+1)=\val_p(j+1)$,  applying Theorem \ref{thm-stirl-valuation}, we get
\begin{align*}
\val_p(\stirlingii{j}{p^{\delta}-1}p^{j-\ell-1}\binom{d_0-\ell}{j-\ell})&\geq
\delta-1-\val_p(j+1)+\val_p(d_0-j+1)\\
&=\delta-1.
\end{align*}

We conclude that $\kappa_{\delta}(u_{\delta})$ vanishes
on $N_{\delta}$.
\end{proof}

We continue to finish the proof of Proposition \ref{prop-lift-fixedelement-divpow}.

Now we get a co-chain 
\[
\kappa_{\delta}: G_1\rightarrow D(V_{\delta-1})^{N_{\delta}}, 
\]
which defines an element in $H^1(G_1,D(V_{\delta-1})^{N_{\delta}})$.

Now we use the special fact about $G_1$: 
the element $T$ generates a cyclic subgroup of 
order $p$ in $G_1$, which is a p-Sylow sub-group.
Therefore if we consider the restriction 
and corestriction morphism of cohomology groups
\[
Res: H^1(G_1, D(V_{\delta-1})^{N_{\delta}})\rightarrow H^1(<T>, D(V_{\delta-1})^{N_{\delta}})
\]
\[
Cor: H^1(<T>, D(V_{\delta-1})^{N_{\delta}})\rightarrow H^1(G_1, D(V_{\delta-1})^{N_{\delta}})
\]
satisfying
\[
Cor\circ Res=[G_1:<T>]=p^2-1.
\]
This shows that $Res$ is actually injective since 
$p^2-1$ is co-prime to $p$.
But we know that
\[
\kappa_{\delta}(u_{\delta})(T)=\frac{T(\nu_{p^\delta -1}-\nu_{p^{\delta+1}-1})-(\nu_{p^\delta -1}-\nu_{p^{\delta+1}-1})}{p}\equiv 0\mod p^{\delta-1}.
\]
Therefore the co-chain $\kappa_{\delta}(u_{\delta})$
must be a co-boundary in $H^1(G_\delta, D(V_{\delta-1}))$.
This proves the existence of lifting.
\end{proof}

\remk 
The whole argument relies on 
the choice of $t_{\delta}$.

\remk
 In fact, in general for $p$ odd and $\delta=2$, one can write down an explicit lifting
\[
w_2=\sum_{j=p^{\delta}-1}^{d_0-p^\delta+1}\stirlingii{j}{p^{\delta}-1}\xi_1^{(d-j)}\xi_2^{(j)}.
\]
However, this is the only case where such formulas
are found. Instead for the general case, our proof
only yields the existence of $w_{\delta}$, no explicit
formula could be extracted from the proof.

As a consequence, we have

\begin{cor}
The module $M_{G_{\delta}}^{\delta}$ contains
a primitive element $$X^{p^{\delta+1}-p^{\delta}+p^{\delta-1}-1}Y^{p^\delta-1}.$$ Furthermore, the image of 
$X^{p^{\delta+1}-p^{\delta}+p^{\delta-1}-1}Y^{p^\delta-1}$ under the canonical projection 
\[
M_{G_{\delta}}^{\delta}\rightarrow M_{G_{r}}^{r}
\]
is primitive for any $1\leq r\leq \delta$.
\end{cor}

\begin{proof}
Indeed, consider the lifting $w_{\delta}$ of $u_{\delta}$, then 
\[
\langle w_{\delta}, X^{p^{\delta+1}-p^{\delta}+p^{\delta-1}-1}Y^{p^\delta-1}\langle =c
\]
is a unit in $\bZ/p^\delta$.
\end{proof}

\begin{lemma}\label{lemma-primitive-element-special-expression}
In $M_{G_1}^1$, we have
\begin{align}\label{equation-monomial-reduction}
&X^{p^{\delta+1}-p^{\delta}+p^{\delta-1}-1}Y^{p^\delta-1}\nonumber\\
&=X^{p^2-p}Y^{p-1}h(f_1, f_2)+\sum_{0\leq\ell<p}X^{ \ell(p-1)}Y^{p-1}a_{\ell}(f_1, f_2)\nonumber \\
&+\sum_{0\leq \ell<p+1} X^{\ell(p-1)}b_{\ell}(f_1, f_2) 
\end{align}
satisfying the conditions 
\begin{description}
\item[(A)]$h(f_1, f_2)=f_1^{p^{\delta-1}-1}f_2^{p^{\delta-1}-1}+\sum_{j<p^{\delta-1}-1}h_{i, j}f_1^if_2^j$;
\item[(B)]$\deg_{f_2}(a_{\ell}(f_1, f_2))\leq p^{\delta-1}-p+l-1 $, for $ 1< \ell<p.$
\item[(C)]$\deg_{f_2}(b_{\ell}(f_1, f_2))\leq p^{\delta-1}-p+l-1$, for $ 1< \ell<p+1.$
\end{description}
\end{lemma}
\begin{proof}[Proof of Lemma \ref{lemma-primitive-element-special-expression}]

We show this by induction on $\delta$.
For $\delta=1$, the left hand side is $X^{p^2-p}Y^{p-1}$. Assume that we have 
the desired expression for $\delta$. 
Then for $\delta+1$, 
\begin{align}\label{equation-monomial-deltaplusone}
X^{p^{\delta+2}-p^{\delta+1}+p^{\delta}-1}Y^{p^\delta-1}&=X^{p(p^{\delta+1}-p^\delta+p^{\delta-1}-1)+p-1}Y^{p(p^\delta-1)+p-1}\nonumber\\
&=(X^{p^{\delta+1}-p^{\delta}+p^{\delta-1}-1}Y^{p^\delta-1})^pX^{p-1}Y^{p-1}.
\end{align}
Applying the induction on $\delta$, the
right hand side of (\ref{equation-monomial-deltaplusone}) equals to
\begin{align*}
&(X^{p(p^2-p)}Y^{p(p-1)}h(f_1^p, f_2^p)+\sum_{0\leq \ell<p}X^{\ell p(p-1)}Y^{p(p-1)}a_{\ell}(f_1^p, f_2^p) \\
&+\sum_{0\leq \ell<p+1} X^{\ell p(p-1)}b_{\ell}(f_1^p, f_2^p))X^{p-1}Y^{p-1},     
\end{align*}
which simplies to be 
\begin{align}\label{equation-intermidiate-mononial-reduction}
&(X^{p(p^2-p)+p-1}Y^{p^2-1}h(f_1^p, f_2^p)+\sum_{0\leq \ell<p}X^{\ell p(p-1)+p-1}Y^{p^2-1}a_{\ell}(f_1^p, f_2^p)\nonumber \\
&+\sum_{0\leq \ell<p+1} X^{\ell p(p-1)+p-1}Y^{p-1}b_{\ell}(f_1^p, f_2^p).
\end{align}
We need the following two relations
between $f_1$ and $f_2$
\begin{align*}
X^{p^2-1}&=X^{p-1}f_2-f_1^{p-1}\\ 
Y^{p^2-1}&=Y^{p-1}f_2-f_1^{p-1}.
\end{align*}
Then 
\begin{align*}
&X^{p(p^2-p)+p-1}Y^{p^2-1}\\
&=X^{p^2(p-1)}Y^{p^2-1}X^{p-1}\\
&=(X^pf_2-Xf_1^{p-1})^{p-1}(Y^{p-1}f_2-f_1^{p-1})X^{p-1}\\
&=(\sum_{i=0}^{p-1}X^{(p-1)(i+1)}f_2^if_1^{(p-1)(p-i-1)})(Y^{p-1}f_2-f_1^{p-1})X^{p-1}\\
&=(X^{p-1}f_2-f_1^{p-1})Y^{p-1}f_2^p+X^{p(p-1)}Y^{p-1}f_2^{p-1}f_1^{p-1}\\
&+\sum_{2\leq\ell\leq p-1}X^{ \ell(p-1)}Y^{p-1}f_2^{\ell-1}f_1^{(p-1)(p-\ell+1)}-\sum_{2\leq \ell\leq p}X^{\ell(p-1)}f_2^{\ell-2}f_1^{(p-1)(p-\ell+2)}\\
&+(X^{p-1}f_2-f_1^{p-1})f_2^{p-1}f_1^{p-1}.
\end{align*}
This shows the term 
\[
(X^{p(p^2-p)+p-1}Y^{p^2-1}f_{1}^{p(p^{\delta-1}-1)}f_{2}^{p(p^{\delta-1}-1)}
\]
in (\ref{equation-intermidiate-mononial-reduction}) is of the form described on the right hand side of (\ref{equation-monomial-reduction}). Moreover, 
for $1<\ell<p$, 
\[
\deg_{f_2}(f_2^{\ell-1}f_1^{(p-1)(p-\ell+1)}h(f_1^p, f_2^p))\leq \ell-1+p(p^{\delta-1}-1)=p^{\delta}-p+\ell-1
\]
and for $1<\ell<p+1$
\[
\deg_{f_2}(f_2^{\ell-2}f_1^{(p-1)(p-\ell+2)}h(f_1^p, f_2^p))\leq \ell-2+p(p^{\delta-1}-1)\leq p^{\delta}-p+\ell-1, 
\]
which shows the conditions (B) and (C).
And for the second term in (\ref{equation-intermidiate-mononial-reduction}), 
\begin{align*}
&X^{\ell p(p-1)+p-1}Y^{p^2-1}a_{\ell}(f_1^p, f_2^p)\\
&=X^{\ell p(p-1)+p-1}Y^{p-1}f_2a_{\ell}(f_1^p, f_2^p)-X^{\ell p(p-1)+p-1}f_1^{p-1}a_{\ell}(f_1^p, f_2^p), 
\end{align*}
we prove by induction on $\ell$ that the term
\[
X^{\ell p(p-1)+p-1}Y^{p-1}f_2a_{\ell}(f_1^p, f_2^p)
\]
is of the form described on right hand side of (\ref{equation-monomial-reduction})
and satisfying (B) and (C).
For $\ell=0$, the term $X^{p-1}Y^{p-1}f_2a_{\ell}(f_1^p, f_2^p)$
is already of the desired form. 
And for $\ell>0$, 
\begin{align*}
&X^{\ell p(p-1)+p-1}Y^{p-1}f_2a_{\ell}(f_1^p, f_2^p)\\
&=X^{(\ell-1) p(p-1)}(X^{p-1}f_2-f_1^{p-1})Y^{p-1}a_{\ell}(f_1^p, f_2^p)\\
&=X^{(\ell-1) p(p-1)+p-1}Y^{p-1}f_2a_{\ell}(f_1^p, f_2^p)-X^{(\ell-1) p(p-1)}Y^{p-1}f_1^{p-1}a_{\ell}(f_1^p, f_2^p).
\end{align*}

If furthermore $\ell\geq 2$, 
\begin{align*}
&X^{(\ell-1) p(p-1)}Y^{p-1}\\
&=X^{(\ell-2) (p^2-1)+(p-1)(p-\ell+2)}Y^{p-1}\\
&=(X^{p-1}f_2-f_1^{p-1})^{\ell-2}X^{(p-1)(p-\ell+2)}Y^{p-1}\\
&=X^{p(p-1)}Y^{p-1}f_{2}^{\ell-2}+\sum_{i<\ell-2}d_iX^{(p-1)(p-\ell+i+2)}Y^{p-1}f_2^i f_1^{(p-1)(\ell-2-i)}
\end{align*}
Therefore for $2\leq \ell\leq p-1$, we get a contribution
\[
X^{p(p-1)}Y^{p-1}f_1^{p-1}f_{2}^{\ell-2}a_{\ell}(f_1^p, f_2^p)
\]
by assumption, we know that 
\[
\deg_{f_2}(f_1^{p-1}f_{2}^{\ell-2}a_{\ell}(f_1^p, f_2^p))\leq p(p^{\delta-1}-p+\ell-1)+\ell-2<p^{\delta}-1.
\]
Also for $i<\ell-2$, we get 
\[
X^{(p-1)(p-\ell+i+2)}Y^{p-1}f_2^if_1^{(p-1)(\ell-i-1)}a_{\ell}(f_1^p, f_2^p)
\]
satisfying 
\begin{align*}
\deg_{f_2}(f_2^if_1^{(p-1)(\ell-i-1)}a_{\ell}(f_1^p, f_2^p))&\leq i+p(p^{\delta-1}-p+\ell-1)\\
&\leq p^{\delta}-p+(p-\ell+i+2)-1, 
\end{align*}
which verifies the condition (B) above. Applying
induction on $\ell$ shows that 
\[
X^{\ell p(p-1)+p-1}Y^{p-1}f_2a_{\ell}(f_1^p, f_2^p)
\]
is also of the form described on the right hand 
side of (\ref{equation-monomial-reduction}).
We still need to show that the term
\[
\sum_{0\leq \ell<p+1} X^{\ell p(p-1)+p-1}Y^{p-1}b_{\ell}(f_1^p, f_2^p)
\]
is of the form on the right hand side of 
(\ref{equation-monomial-reduction}) and satisfying (B) and (C).
The case of $0\leq \ell\leq p-1$ is already proved. 
For $l=p$, we have
\begin{align*}
&X^{p^2(p-1)+p-1}Y^{p-1}\\
&=(X^pf_2-Xf_1^{p-1})^{p-1}X^{p-1}Y^{p-1}\\
&=(\sum_{i=0}^{p-1}X^{(p-1)(i+1)}f_2^if_1^{(p-1)(p-i-1)})X^{p-1}Y^{p-1}\\
&=(X^{p-1}f_2-f_1^{p-1})Y^{p-1}f_2^{p-1}+X^{p(p-1)}Y^{p-1}f_2^{p-2}f_1^{p-1}\\
&+\sum_{2\leq r<p}X^{ r(p-1)}Y^{p-1}f_2^{r-2}f_1^{(p-1)(p-r+1)}.
\end{align*}
And we have
\begin{align*}
&\deg_{f_2}(f_2^{p-2}f_1^{p-1}b_{p}(f_1^p, f_2^p))\leq p(p^{\delta-1}-p+p-1)+p-2<p^{\delta}-1, 
\end{align*}
and
\begin{align*}
\deg_{f_2}(f_2^{r-2}f_1^{(p-1)(p-r+1)}b_{p}(f_1^p, f_2^p))&\leq p(p^{\delta-1}-p+p-1)+r-2\\
&  \leq p^{\delta}-p+r-1.
\end{align*}
This proves the lemma.
\end{proof}
\begin{lemma}\label{lemma-pairing-lifting-multiplication-nonvanishing}
Assume $p>3$ be a prime. Let 
\[
U_{\delta}:=X^{p^{\delta+1}-p^{\delta}+p^{\delta-1}-1}Y^{p^\delta-1}.
\]
For $r\geq 1$, we have
\[
\langle w_{\delta+r}, f_{1, \delta}^{p^{r}-1}f_{2, \delta}^{p^{r}-1}U_{\delta}\rangle 
\]
is a unit in $\bZ/p^{\delta}$, where $w_{\delta+r}$ is considered to be the primitive element in 
$D(V_{\delta})^{G_{\delta}}$. Moreover, for $a>0$, 
\[
\langle w_{\delta+r}, f_{1, \delta}^{p^{r}-1+ap(p-1)}f_{2, \delta}^{p^{r}-1-a(p+1)}U_{\delta}\rangle\equiv 0, \quad \mod p, 
\]
where 
$a\in 1/2\bZ$.
\end{lemma}

\remk For an element
$f_{1, \delta}^{p^{r}-1-b_2}f_{2, \delta}^{p^{r}-1+b_1}$ to be of
the same degree as $f_{1, \delta}^{p^{r}-1}f_{2, \delta}^{p^{r}-1}$, we must have
\[
b_1=a(p+1), \quad b_2=ap(p-1), 
\]
with $a$ being half integer for $p\geq 3$.

\begin{proof}[Proof of Lemma \ref{lemma-pairing-lifting-multiplication-nonvanishing}]
To show 
\[
\langle w_{\delta+r}, f_{1, \delta}^{p^{r}-1}f_{2, \delta}^{p^{r}-1}U_{\delta}\rangle 
\]
is a unit in $\bZ/p^{\delta}$, we can take 
the projection onto $\bF_p$. Then it 
suffice to show that
\[
\langle u_{\delta+r}, f_{1, \delta}^{p^{r}-1}f_{2, \delta}^{p^{r}-1}U_{\delta}\rangle 
\]
is a unit in $\bF_p$. But over $\bF_p$, we have
\begin{align*}
&f_{1, \delta}^{p^{r}-1}f_{2, \delta}^{p^{r}-1}U_{\delta}\\
&=(X^{p^2}Y-XY^{p^2})^{p^{\delta-1}(p^{r}-1)}U_{\delta}\\
&=(X^{p^{\delta+1}}Y^{p^{\delta-1}}-X^{p^{\delta-1}}Y^{p^{\delta+1}})^{p^{r}-1}X^{p^{\delta+1}-p^{\delta}+p^{\delta-1}-1}Y^{p^{\delta}-1}\\
&=\sum_{j=0}^{p^{r}-1}X^{jp^{\delta+1}+(p^r-1-j)p^{\delta-1}}Y^{jp^{\delta-1}+(p^{r}-j-1)p^{\delta+1}}
X^{p^{\delta+1}-p^{\delta}+p^{\delta-1}-1}Y^{p^{\delta}-1}\\
&=\sum_{j=0}^{p^{r}-1}X^{p^{\delta+r-1}-1+p^{\delta-1}(p-1)(p+j(p+1))}Y^{p^{\delta+r+1}-1-p^{\delta-1}(p-1)(p+j(p+1))}.
\end{align*}
And we have
\[
u_{\delta+r}=\sum_{\ell=1}^{p}\xi_1^{(p^{\delta+r-1}-1+\ell p^{\delta+r-1}(p-1))}
\xi_2^{(p^{\delta+r-1}-1+(p-\ell+1)p^{\delta+r-1}(p-1))}.
\]
Therefore we have
$\langle u_{\delta+r}, f_{1, \delta}^{p^{r}-1}f_{2, \delta}^{p^{r}-1}U_{\delta}\rangle $
equal to 
\[
\sharp\{j|0\leq j\leq p^r-1: p+j(p+1)=\ell p^r\text{ for some } 1\leq \ell \leq p\}.
\]
For $r$ odd and $\ell=1$, we know 
\[
j=\frac{p(p^{r-1}-1)}{p+1}
\]
is an integer. And for $2\leq \ell\leq p$, 
\[
p+\ell=\ell(p^r+1)-j(p+1)\equiv 0, \quad \mod (p+1) 
\]
admits no solution. For $r$ even and $\ell=p$, 
\[
j=\frac{p(p^{r}-1)}{p+1}
\]
is an integer. And for $1\leq \ell\leq p-1$, 
\[
p+\ell p=\ell p(p^{r-1}+1)-j(p+1)\equiv 0, \quad \mod (p+1)
\]
admits no solution. Therefore
\[
\langle u_{\delta+r}, f_{1, \delta}^{p^{r}-1}f_{2, \delta}^{p^{r}-1}U_{\delta}\rangle=1.
\]
Let $a>0$, 
\begin{align*}
&f_{1, \delta}^{p^{r}-1+ap(p-1)}f_{2, \delta}^{p^{r}-1-a(p+1)}U_{\delta}\\
&=(X^pY-XY^p)^{p^{\delta-1}a(p^2+1)}(X^{p^2}Y-XY^{p^2})^{p^{\delta-1}((p^{r}-1)-a(p+1))}U_{\delta}\\
&=(X^{p^{\delta}}Y^{p^{\delta-1}}-X^{p^{\delta}-1}Y^{p^{\delta}})^{a(p^2+1)}(X^{p^{\delta+1}}Y^{p^{\delta-1}}-X^{p^{\delta-1}}Y^{p^{\delta+1}})^{p^{r}-1-a(p+1)}U_{\delta}\\
&=(\sum_{i=0}^{a(p^2+1)}(-1)^{a(p^2+1)-i}\binom{a(p^2+1)}{i}X^{ip^{\delta}+(a(p^2+1)-i)p^{\delta-1}}Y^{p^{\delta-1}+(a(p^2+1)-i)p^{\delta}})\\
&\hspace{0.5cm}(\sum_{j=0}^{p^r-1-a(p+1)}(-1)^{p^r-1-a(p+1)-j}\binom{p^r-1-a(p+1)}{j}X^{jp^{\delta+1}+(p^r-1-j-a(p+1))p^{\delta-1}}\\
&\hspace{0.5cm}Y^{jp^{\delta-1}+(p^r-1-j-a(p+1))p^{\delta+1}})U_{\delta}\\
&=\sum_{i=0}^{a(p^2+1)}\sum_{j=0}^{p^r-1-a(p+1)}(-1)^{p^r-1+a(p^2-p)-i-j}\binom{a(p^2+1)}{i}\binom{p^r-1-a(p+1)}{j}\\
&\hspace{0.5cm}X^{p^{\delta+r-1}-p^{\delta-1}+p^{\delta-1}(ip-i+p^2j-j+ap(p-1))}Y^{p^{\delta+r+1}-p^{\delta+1}+p^{\delta-1}(-ip+i-p^2j+j-ap(p-1))}U_{\delta}\\
&=\sum_{i=0}^{a(p^2+1)}\sum_{j=0}^{p^r-1-a(p+1)}(-1)^{p^r-1+a(p^2-p)-i-j}\binom{a(p^2+1)}{i}\binom{p^r-1-a(p+1)}{j}\\
&\hspace{0.5cm}X^{p^{\delta+r-1}-1+p^{\delta-1}(p-1)((j+a)p+i+j+p)}Y^{p^{\delta+r+1}-1-p^{\delta-1}(p-1)((j+a)p+i+j+p)}.
\end{align*}
Again, recall that
\[
u_{\delta+r}=\sum_{\ell=1}^{p}\xi_1^{(p^{\delta+r-1}-1+\ell p^{\delta+r-1}(p-1))}
\xi_2^{(p^{\delta+r-1}-1+(p-\ell +1)p^{\delta+r-1}(p-1))}.
\]
Hence we need to consider the equation
\[
p^{\delta-1}(p-1)((j+a)p+i+j+p)=\ell p^{\delta+r-1}(p-1)
\]
or equivalently, 
\begin{equation}\label{equation-binomial-pairing-relation}
p^r\ell=j(p+1)+ap+i+p.
\end{equation}

In fact, when $p$ is odd, the equation (\ref{equation-binomial-pairing-relation}) requires $a$ to be integer. From now on we consider 
$a$ to be integer and prove the general case.
The equation (\ref{equation-binomial-pairing-relation}) allows
us to reduce to prove that the sum of coefficients
$\sum_{\ell=1}^pb_{\ell p^r}$ in 
\[
Q(t):=t^{(a+1)p}(1-t)^{a(p^2+1)}(1-t^{p+1})^{p^r-1-a(p+1)}=\sum_{i=0}^{\deg(Q(t))}b_it^i
\]
vanishes $\bF_p$, where $b_i$ are all integers. We have
\[
\deg(Q(t))=(a+1)p+a(p^2+1)+(p+1)(p^r-1-a(p+1))=p^{r+1}+p^r-ap-1
\]
As a consequence, we see that 
\[
\sum_{p^r\mid i}b_i=\sum_{\ell=1}^pb_{\ell p^r}.
\]
Motivated by this, we define for any 
$h(t)=\sum_{i=0}^{d}c_it^i\in \bF_p[t]$, 
\[
S_r(h)=\sum_{p^r\mid i}c_i\in \bF_p.
\]
We observe that the operator $S_r$
is linear and invariant under multiplication by $t^{p^r}$, i.e, 
\[
S_r(ht^{p^r})=S_r(h).
\]
One gets as an immediate consequence that
\[
S_r(h(t^{p^r}-1))=0.
\]
Now we only need to show that
\[
Q(t)\equiv =(1-t^{p^r})Q'(t), \quad \mod p.
\]
Over $\bF_p$, 
\[
(1-t^{p^r})=(1-t)^{p^r}
\]
and $Q(t)$ is divisible by
\[
(1-t)^{a(p^2+1)}(1-t)^{p^r-1-a(p+1)}=(1-t)^{p^r-1+a(p^2-p)}
\]
the condition that $a>0$ shows the desired result.
\end{proof}

\remk 
Note that we have for $\zeta_{p^r}$ the $p^r$-th roots
of unity, 
\[
\sum_{k=0}^{p^r-1}\zeta_{p^r}^{ck}=\left\{\begin{array}{lcr}
p^r ,&  \text{ if }  p^r\mid c, \\
0 ,&  \text{ otherwise }.
\end{array}\right. 
\]
This shows that consider $Q(t)$ as an element in $\bZ[t]$, we have
\[
\sum_{\ell=1}^{p}b_{p^r\ell}=\frac{\sum_{i=0}^{p^r-1}Q(\zeta_{p^r})}{p^r}.
\]
As a consequence, we get
\[
\val_p(\sum_{i=1}^pQ(\zeta_{p^r}))\geq r+1.
\]

\remk Indeed the statement also holds for 
 $a<0$, but we do not need it.
 
\remk We thank 
Danylo Radchenko for discussion on the last part of the proof.

\begin{prop}\label{prop-primitive-element-stable-multiplication}
Let $f\in M^{\delta, G_{\delta}}_{\prim}=\bZ/p^{\delta}[f_{1, \delta}, f_{2, \delta}]$ such that
$f\neq 0 \mod p$. Then the element $fX^{p^{\delta+1}-p^{\delta}+p^{\delta-1}-1}Y^{p^\delta-1}$ is also primitive.
\end{prop}

\begin{proof}
We first show that this is the case when 
$f$ is a monomial. For simplicity, 
we denote 
\[
U_{\delta}:=X^{p^{\delta+1}-p^{\delta}+p^{\delta-1}-1}Y^{p^\delta-1}.
\]

The lemma above implies that 
the element $f_{1, \delta}^{p^{r}-1}f_{2, \delta}^{p^{r}-1}U_{\delta}$ is primitive
in $M^{\delta}_{G_{\delta}}$. 
Therefore for any monomial
$f=f_{1, \delta}^af_{2, \delta}^{b}$, we can choose
$f'=f_{1, \delta}^{p^r-a-1}f_{2, \delta}^{p^r-b-1}$, where
$r$ is an integer such that $p^r-1\geq \max\{a, b\}$.
Then $f'fU_{\delta}$ is primitive, which implies that
$fU_{\delta}$ itself is primitive.
For general case let 
$f=\sum_{i}h_i$ be a linear combination of monomials $h_i=c_if_{1, \delta}^{a_i}f_{2, \delta}^{b_i}$. 
Without loss of generality, assume that $c_i\neq 0 \mod p$ for all $i$, and
\[
b_1>b_2>\cdots, \quad a_1+b_1=a_2+b_2=\cdots
\]
Let $r$ be the minimal integer such that
$p^r-1\geq \max\{a_1, b_1\}$.
Hence 
\[
f_{1, \delta}^{p^r-1-a_1}f_{2, \delta}^{p^r-1-b_1}fU_{\delta}=f_{1, \delta}^{p^r-1-a_1}f_{2, \delta}^{p^r-1-b_1}h_1U_{\delta}+f_{1, \delta}^{p^r-1-a_1}f_{2, \delta}^{p^r-1-b_1}h_2U_{\delta}+\cdots.
\]
We claim that
\[
\langle u_{\delta+r},f_{1, \delta}^{p^r-1-a_1}f_{2, \delta}^{p^r-1-b_1}fU_{\delta}\rangle 
\]
is nonzero in $\bF_p$. Note that
in the lemma above, we have shown
\[
\langle u_{\delta+r},f_{1, \delta}^{p^r-1-a_1}f_{2, \delta}^{p^r-1-b_1}h_1U_{\delta}\rangle =1.
\]
and for $i>1$, 
\[
\langle u_{\delta+r},f_{1, \delta}^{p^r-1-a_1}f_{2, \delta}^{p^r-1-b_1}h_iU_{\delta}\rangle =0.
\]
This shows that $fX^{p^{\delta+1}-p^{\delta}+p^{\delta-1}-1}Y^{p^\delta-1}$ is primitive.
\end{proof}

\begin{cor}\label{cor-second-cohomology-nonvanishing-p-powers-freeness}
 Let $M^{\delta, G_{\delta}}_{\prim}=\bZ/p^{\delta}[f_{1, \delta}, f_{2, \delta}]$. Then 
 \[
\Ann_{M^{\delta, G_{\delta}}_{\prim}}(X^{p^{\delta+1}-p^{\delta}+p^{\delta-1}-1}Y^{p^\delta-1})=0
 \]
 Therefore $X^{p^{\delta+1}-p^{\delta}+p^{\delta-1}-1}Y^{p^\delta-1}$ generates a free module of rank one over $M^{\delta, G_{\delta}}_{\prim}$.
\end{cor}

\begin{proof}
We recall that
\begin{align*}
f_{1, \delta}&=(X^pY-XY^p)^{p^{\delta-1}}\\
f_{2, \delta}&=(X^{p(p-1)}+X^{(p-1)(p-1)}Y^{p-1}+\cdots+Y^{p(p-1)})^{p^{\delta-1}}.
\end{align*}
Suppose that $f\in M^{\delta, G_{\delta}}_{\prim}=\bZ/p^{\delta}[f_{1, \delta}, f_{2, \delta}]$ annihilates the element 
$X^{p^{\delta+1}-p^{\delta}+p^{\delta-1}-1}Y^{p^{\delta}-1}$.
Assume $f=p^rf_2$ for some $r\geq 0$ and $f_2\neq 0\mod p$.
But then Proposition \ref{prop-primitive-element-stable-multiplication} shows that
$f_2X^{p^{\delta+1}-p^{\delta}+p^{\delta-1}-1}Y^{p^{\delta}-1}$
is of order $p^\delta$. This implies $r\geq \delta$, therefore
$f=0$ in $M^{\delta, G_{\delta}}_{\prim}$.
This shows 
\[
\Ann_{M^{\delta, G_{\delta}}_{\prim}}(X^{p^{\delta+1}-p^{\delta}+p^{\delta-1}-1}Y^{p^{\delta}-1})=0.
\]

By Proposition \ref{prop-divided-power-struture-modp}, we 
known that the sub-module $N$ of $M^1_{G_1}$
generated by $X^{p^2-p}Y^{p-1}$ is free of rank
one over $M^{1, G_1}$. Lemma \ref{lemma-primitive-element-special-expression} above shows that
the element 
$X^{p^{\delta+1}-p^{\delta}+p^{\delta-1}-1}Y^{p^{\delta}-1}$
is non-trivial under the projection of $M^1_{G_1}$
to $N$, this shows the freeness of 
$$M^{\delta, G_{\delta}}_{\prim}X^{p^{\delta+1}-p^{\delta}+p^{\delta-1}-1}Y^{p^{\delta}-1}.$$
\end{proof}

We return to the proof of Proposition \ref{prop-second-cohomology-module-structures}.

\begin{lemma}\label{lemma-vanishing-diagonal-action}
Let $d>0$. Then the monomial $X^{a}Y^{b}$ vanishes in $M^\delta_{G_\delta}$ if $p-1\nmid (a-b)$. Moreover, 
if $p-1\mid (a-b)$, let $r=\val_p(\frac{a-b}{p-1})$, then
$X^{a}Y^{b}\in M^\delta_{G_\delta}[p^{r+1}]$, where
$M^\delta_{G_\delta}[p^{r+1}]$ is the sub-module
generated by elements killed by $p^{r+1}$.
\end{lemma}
\begin{proof}
In fact, let $g=\begin{pmatrix}c&0\\ 0& c^{-1}\end{pmatrix}\in G_\delta$ with $c\in (\bZ/p^\delta)^{\times}_p$.
Then 
\[
(\Id-g)(X^aY^b)=(1-c^{a-b})X^aY^{b}.
\]
If $p-1\nmid a-b$, then picking $c$ with $c^{a-b}\neq 1 \mod p$ yields the result. And if $p-1 \mid a-b$, and
\[
r=\val_p(\frac{a-b}{p-1}). 
\]
Then $p^r(p-1)\mid (a-b)$. Pick $c\in \bZ/p^{\delta}$, 
such that
\[
1-c^{a-b}=c_0 p^{r+1}, c_0\in (\bZ/p^{\delta})^{\times}.
\]
This shows the result.

\end{proof}

We first take care of the case $\delta=2$.
\begin{lemma}\label{lemma-primitive-element-non-existence-two}
Let $p>3$. We have
\begin{description}
\item[(1)]For $0<d<p^3+p-2$, no elements in $M^2_{G_2, d}$
are primitive.
\item[(2)]For $d=p^3+p-2$, all primitive elements $f$ in $M^2_{G_2, d}$ are of the form $$cX^{p^3-p^2+p-1}Y^{p^2-1}+h, $$ where
$c$ is a unit and $h\in M^2_{G_2}[p]$, where $M^2_{G_2}[p]$
 is the sub-module of elements of killed by $p$ in $M^2_{G_2}$. 

\item[(3)]For $d>p^3+p-2$, all primitive elements are  
 of the form $$cf_{1, 2}^af_{2, 2}^bX^{p^3-p^2+p-1}Y^{p^2-1}+h, $$ where
$c$ is a unit,  $h\in M^2_{G_2}[p]$ and
\[
d=ap(p+1)+bp^2(p-1)+p^3+p-2.
\]
\end{description}
\end{lemma}
\remk The case of $d=0$ is omitted for 
triviality.

\begin{proof}

We first assume that $d<p^3+p-2=p(p-1)^2+p^2-1$.
Recall that we have a canonical projection
\[
\pi_2: M^{2}/(\Id-T)M^2\rightarrow M^2_{G_2}
\]
and 
\begin{align}\label{equation-decomposition-primitive-boundary-module
}
(M^{2}/(\Id-T)M^2)_{d, \prim}=\bZ/p^2 Y^d \oplus\bigoplus_{k, p^2\mid k+1} \bZ/p^2 \epsilon_k.
\end{align}
Note that
\[
Y^d=(S-\Id)X^d+X^d
\]
and $X^d$ vanishes in $M^{2}/(\Id-T)M^2$, therefore
$Y^d$ vanishes in $M^2_{G_2}$. 
We are reduced to show that all elements $\epsilon_{\ell p^2-1}$
lie in $M^2_{G_2}[p]$ for $\ell\leq p$. 
Note that for
$p>2$, 
by Lemma 5.1 of \cite{CH10}, we have
\begin{align}\label{equation-stirling-power-series-p-square}
\sum_{n=\ell p^2-1}^{\infty}\stirlingii{n}{\ell p^2-1}t^n&=\frac{t^{\ell p^2-1}}{(1-t)(1-2t)\cdots(1-(\ell p^2-1)t)}\nonumber\\
&\equiv \frac{t^{\ell p^2-1}}{(1-t^{p-1})^{p\ell}}
\quad \mod p^2.
\end{align}
From this equality we get for $1\leq \ell'<\ell<p$, 
\begin{align}\label{equation-stirling-number-psquare}
\stirlingii{\ell p^2-1}{\ell' p^2-1}\equiv 0, \quad \mod p^2.
\end{align}
and 
\begin{align}\label{equation-stirling-number-psquare-two}
\stirlingii{(\ell+r) p^2-rp-1}{\ell' p^2-1}\equiv 0, \quad \mod p^2.
\end{align}
\begin{align}\label{equation-stirling-number-psquare-nonvanishing-two}
\stirlingii{(\ell+r) p^2-rp-1}{\ell p^2-1}\equiv \binom{(\ell+r)p-1}{\ell p-1}, \quad \mod p^2.
\end{align}
Note that we have
\[
X^{d-\ell p^2+1}Y^{\ell p^2-1}=\epsilon_{\ell p^2-1}+\sum_{r<\ell p^2-1}\stirlingii{\ell p^2-1}{r}\epsilon_r.
\]

And applying equation (\ref{equation-stirling-number-psquare}) yields in $M^2_{G_2}$, 
\[
pX^{d-\ell p^2+1}Y^{\ell p^2-1}=p\epsilon_{\ell p^2-1}, 
\]
 And lemma \ref{lemma-vanishing-diagonal-action} shows that
$pX^{d-\ell p^2+1}Y^{\ell p^2-1}\equiv 0$ unless 
$p(p-1)\mid d-2\ell p^2+2$. Now let
\[
d=2\ell p^2-2+kp(p-1).
\]
Now 
\[
(T-\Id)(X^{d-\ell p^2-p+1}Y^{\ell p^2+p-1})=\sum_{i=0}^{\ell p^2+p-2}\binom{\ell p^2+p-1}{i} X^{d-i} Y^i.
\]
Note that for $\ell p^2-1<i\leq \ell p^2+p-2$, we have
$p(p-1)\nmid d-2i$, therefore $X^{d-i}Y^i$ vanishes in 
$M^2_{G_2}$. Moreover, 
\begin{align*}
&(T-\Id)(X^{d-\ell p^2-p+1}Y^{\ell p^2+p-1})\\
&\equiv X^{d-\ell p^2-p+1}(X^{p^2}+Y^{p^2})^{\ell}(X+Y)^{p-1}-X^{d-\ell p^2-p+1}Y^{\ell p^2+p-1}\mod p
\end{align*}
which implies that if $\binom{\ell p^2+p-1}{i}\not\equiv 0\mod p$  then 
$i=i_1p^2+i_2$ with 
$0\leq i_1\leq \ell $, $0\leq i_2\leq p-1$. In this case
\[
d-2i=2((\ell-i_1)p^2-i_2-1)+kp(p-1), 
\]
and $p(p-1)\mid d-2i$ implies $i_2=p-1$ and for $p>3$
\[
i_1=
\ell-\frac{p+1}{2}\text{ or }\ell-1 
\]
We prove by induction on $\ell$ that 
\[
p X^{d-(\ell-1)p^2-(p-1)}Y^{(\ell-1)p^2+p-1}=0
\]
in $M^2_{G_2}$. Note that for $\ell<\frac{p+1}{2}$, we have
\[
p(T-\Id)(X^{d-\ell p^2-p+1}Y^{\ell p^2+p-1})
=p a_\ell X^{d-(\ell-1)p^2-(p-1)}Y^{(\ell-1)p^2+p-1}, 
\]
for some $a_\ell\in (\bZ/p^2)^{\times}.$
This proves 
\[
p X^{d-(\ell-1)p^2-(p-1)}Y^{(\ell-1)p^2+p-1}=0.
\]
And for $\ell\geq \frac{p+1}{2}$, 
\begin{align*}
&p(T-\Id)(X^{d-\ell p^2-p+1}Y^{\ell p^2+p-1})\\
&=p a_\ell X^{d-(\ell-1)p^2-(p-1)}Y^{(\ell-1)p^2+p-1}\\
&+pb_{\ell}X^{d-(\ell-\frac{p+1}{2})p^2-(p-1)}Y^{(\ell-\frac{p+1}{2})p^2+p-1}, 
\end{align*}
for some units $a_\ell, b_\ell \in (\bZ/p^2)^\times$. 
Apply induction shows the result.
Now this implies
\begin{align*}
pX^{(\ell-1)p^2+p-1}Y^{d-(\ell-1)p^2-(p-1)}
&=(S-\Id)(pX^{d-(\ell-1)p^2-(p-1)}Y^{(\ell-1)p^2+p-1})\\
&+pX^{d-(\ell-1)p^2-(p-1)}Y^{(\ell-1)p^2+p-1}\\
&=0
\end{align*}
We deduce from the condition
\[
d=2\ell p^2-2+kp(p-1)<p^3+p-2
\]
that
\[
k<p
\]
and
\[
p> 2\ell+k-\frac{k+1}{p}\geq \ell+k+\ell-\frac{k+1}{p}\geq \ell+k
\]
And by Lemma \ref{lemma-Lucas-theorem}, we have for $\ell+k<p$, 
\begin{align}\label{equation-binomial-unit-modulo-p-square}
\binom{(k+\ell+1)p-1}{\ell p-1}\equiv \binom{k+\ell}{\ell-1},  \quad \mod p
\end{align}
which shows that $\binom{(k+\ell+1)p-1}{\ell p-1}$ is a unit in $\bZ/p^2$.
Now 
\[
X^{(\ell-1)p^2+p-1}Y^{d-(\ell-1)p^2-(p-1)}=\sum_{r\leq d-(\ell-1)p^2-(p-1)}\stirlingii{d-(\ell-1)p^2-(p-1)}{r}\epsilon_r.
\]
and 
\[
d-(\ell-1)p^2-(p-1)=\ell p^2-1+(k+1)p(p-1).
\]
Applying (\ref{equation-stirling-number-psquare-two}) and (\ref{equation-stirling-number-psquare-nonvanishing-two}) gives in $M^2_{G_2}$, 
\begin{align}\label{equation-monomial-muliplication-p-vanishing}
pX^{(\ell-1)p^2+p-1}Y^{d-(\ell-1)p^2-(p-1)}=p\epsilon_{\ell p^2-1+(k+1)p(p-1)}+p\binom{k+\ell}{\ell-1}\epsilon_{\ell p^2-1}.
\end{align}
The fact that $\ell+k<p$ shows that
$k<p-1$, and
\[
p^2\nmid \ell p^2+(k+1)p(p-1)
\]
hence $\epsilon_{\ell p^2-1+(k+1)p(p-1)}$ is killed by $p$, i.e., 
\[
p\epsilon_{\ell p^2-1+(k+1)p(p-1)}=0.
\]
This shows $p\epsilon_{\ell p^2-1}=0$ in $M^2_{G_2}$. To finish the proof of (1), 
we are left to consider the case
$\ell=p$ (the condition $\ell p^2<p^3+p-2$ implies $\ell\leq p$). The equation (\ref{equation-stirling-power-series-p-square}) gives 
\[
\stirlingii{p^3-1}{\ell' p^2-1}=\left\{\begin{array}{lcr}
\binom{2p^2-1}{p^2-1}, &\text{ if } \ell'=1,  \\
0, &\text{ for }1<\ell'<p.
\end{array}\right.
\]
And Lemma \ref{lemma-Lucas-theorem} (Lucas's theorem) shows that
\[
\binom{2p^2-1}{p^2-1}\equiv \binom{2p-1}{p-1}\equiv 0, 
\quad \mod p.
\]
So we still have
\[
pX^{d- p^3+1}Y^{ p^3-1}=p\epsilon_{ p^3-1}.
\]
And the same argument as the case $\ell<p$ holds as long as
we have $p+k<p$, i.e., k<0. This finishes the proof of (1).

As for $d=p^3+p-2$,  again we need to consider the element 
$\epsilon_{\ell p^2-1}$ with $1\leq \ell\leq p$. First
consider the case $1\leq \ell<p$, then as before, we still have
\[
pX^{d-\ell p^2+1}Y^{\ell p^2-1}=p\epsilon_{\ell p^2-1}, 
\]
Furthermore, 
\[
p(p-1)\mid d-2\ell p^2+2
\]
implies for that $p>3$
\[
\ell=
1\text{ or } \frac{p+1}{2}.
\]
Therefore we are reduce to consider the cases
$\ell=1,\frac{p+1}{2}, p$. Note that 
\[
d=p^2-1+p(p-1)^2.
\]
Then $\ell=1$ and $k=p-1$, and equation (\ref{equation-binomial-unit-modulo-p-square}) becomes
\[
\binom{(k+\ell+1)p-1}{\ell p-1}\equiv 1,  \quad \mod p
\]
which is still a unit.
And (\ref{equation-monomial-muliplication-p-vanishing}) becomes
\[
0=pX^{p-1}Y^{p^3-1}=p\epsilon_{p^3-1}+p\epsilon_{ p^2-1}.
\]
But we know that 
\[
X^{p^3-p^2+p-1}Y^{p^2-1}=\epsilon_{p^2-1}+\sum_{r<p^2-1}\stirlingii{p^2-1}{r} \epsilon_{r}
\]
is primitive, hence so is $\epsilon_{p^2-1}$. 
We conclude that $\epsilon_{p^3-1}+\epsilon_{ p^2-1}$ is of order
$p$. Finally for $\ell=\frac{p+1}{2}$, the (\ref{equation-monomial-muliplication-p-vanishing}) becomes
\[
0=pX^{\frac{p^2(p-1)}{2}+p-1}Y^{\frac{p^2(p+1)}{2}-1}=p\epsilon_{\frac{p^2(p+1)}{2}-1}.
\]
This finishes the proof of (2).

As for (3), we consider the action of 
\[
M^{2, G_2}_{\prim}=\bZ/p^2[f_{1, 2}, f_{2, 2}].
\]

We first show the following lemma
\begin{lemma}\label{lemma-quotient-vanishing-higer-degree-p-square}
We have
\[
\oplus_{d>p^3+p-2}(M^1_{G_1}/(f_{1,2},f_{2,2})M^1_{G_1})_d=0
\]
\end{lemma}

\begin{proof}[Proof of lemma \ref{lemma-quotient-vanishing-higer-degree-p-square}]
By definition, we have
\[
f_{1,2}=f_1^{p}, f_{2, 2}=f_{2}^{p}.
\]
By Proposition \ref{prop-divided-power-struture-modp}, we 
know that 
 $M^1_{G_1}/(f_{1,2},f_{2,2})M^1_{G_1}$ 
 as $\bF_p$-vector space, is generated by 
\[
f_2^{a_0}, f_2^{a_1}X^{p-1}Y^{p-1}, f_2^{a_2}X^{2(p-1)}Y^{p-1}, \cdots,f_2^{a_{p-2}}X^{(p-2)(p-1)}Y^{p-1},
f_1^bf_2^{a_{p}}X^{p(p-1)}Y^{p-1}
\]
with $0\leq a_i\leq p-1$ and $0\leq b\leq p-1$.
All of these generators are of degree $\leq p^3+p-2$.
\end{proof}
Now (3) is a consequece of Lemma \ref{lemma-quotient-vanishing-higer-degree-p-square}. We are done.
\end{proof}

\remk The reader will find that the case of $\delta\geq 3$ is exactly the same as $\delta=2$ but the notations are more complicated. 
We still need to treat the case when $\delta>2$.

\begin{lemma}\label{lemma-primitive-element-non-existence-higer-powers}
Let $p>3$. We have
\begin{description}
\item[(1)]For $0<d<p^{\delta+1}+p^{\delta-1}-2$, no elements in $M^\delta_{G_\delta, d}$
are primitive.
\item[(2)]For $d=p^{\delta+1}+p^{\delta-1}-2$, all primitive elements $f$ in $M^\delta_{G_\delta, d}$ are of the form 
\[
cX^{p^{\delta+1}-p^{\delta}+p^{\delta-1}-1}Y^{p^{\delta}-1}+h,
\]
where
$c$ is a unit and $h\in M^\delta_{G_\delta}[p^{\delta-1}]$, where $M^\delta_{G_\delta}[p^{\delta-1}]$
 is the sub-module of elements of killed by $p^{\delta-1}$ in $M^\delta_{G_\delta}$. 
\item[(3)]For $d>p^{\delta+1}+p^{\delta-1}-2$, all primitive elements are  
 of the form 
 \[
 cf_{1, \delta}^af_{2, \delta}^bX^{p^{\delta+1}-p^{\delta}+p^{\delta-1}-1}Y^{p^{\delta}-1}+h,
 \]
 where
$c$ is a unit,  $h\in M^\delta_{G_\delta}[p^{\delta-1}]$ and
\[
d=ap^{\delta-1}(p+1)+bp^{\delta}(p-1)+p^{\delta+1}+p^{\delta-1}-2.
\]
\end{description}
\end{lemma}
\begin{proof}
The case of $\delta=2$ is already treated. Assume from now on $\delta\geq 3$.
We follow the strategy in the proof of case $\delta=2$.

We first assume that $d<p^{\delta+1}+p^{\delta-1}-2$.
Recall that we have a canonical projection
\[
\pi_\delta: M^{\delta}/(\Id-T)M^\delta\rightarrow M^\delta_{G_\delta}
\]
and 
\begin{align}\label{equation-decomposition-primitive-boundary-module-p-power}
(M^{\delta}/(\Id-T)M^\delta)_{d, \prim}=\bZ/p^\delta Y^d \oplus\bigoplus_{r, p^\delta\mid r+1} \bZ/p^\delta \epsilon_r.
\end{align}
Note that
\[
Y^d=(S-\Id)X^d+X^d
\]
and $X^d$ vanishes in $M^{\delta}/(\Id-T)M^\delta$, therefore
$Y^d$ vanishes in $M^\delta_{G_\delta}$. 
We are reduced to show that all elements $\epsilon_{\ell p^\delta-1}$
lie in $M^\delta_{G_\delta}[p^{\delta-1}]$ for $\ell\leq p$. 
Note that for
$p>2$, 
by Lemma 5.1 of \cite{CH10}, we have
\begin{align}\label{equation-stirling-power-series-p-power}
\sum_{n=\ell p^\delta-1}^{\infty}\stirlingii{n}{\ell p^\delta-1}t^n&=\frac{t^{\ell p^\delta-1}}{(1-t)(1-2t)\cdots(1-(\ell p^\delta-1)t)}\nonumber\\
&\equiv \frac{t^{\ell p^\delta-1}}{(1-t^{p-1})^{\ell p^{\delta-1}}}
\quad \mod p^\delta.
\end{align}
From this equality we get for $1\leq \ell'<\ell<p$, 
\begin{align}\label{equation-stirling-number-ppower-different-ell}
\stirlingii{\ell p^\delta-1}{\ell' p^\delta-1}\equiv 0, \quad \mod p^\delta.
\end{align}
and 
\begin{align}\label{equation-stirling-number-p-power-vanishing-differ-r}
\stirlingii{(\ell+r) p^\delta-rp^{\delta-1}-1}{\ell' p^\delta-1}\equiv 0, \quad \mod p^\delta.
\end{align}
\begin{align}\label{equation-stirling-number-ppower-nonvanishing-higher}
\stirlingii{(\ell+r) p^\delta-rp^{\delta-1}-1}{\ell p^\delta-1}\equiv \binom{(\ell+r)p^{\delta-1}-1}{\ell p^{\delta-1}-1}, \quad \mod p^\delta.
\end{align}
Note that we have
\[
X^{d-\ell p^\delta+1}Y^{\ell p^\delta-1}=\epsilon_{\ell p^\delta-1}+\sum_{r<\ell p^\delta-1}\stirlingii{\ell p^\delta-1}{r}\epsilon_r.
\]

And applying equation (\ref{equation-stirling-number-ppower-different-ell}) yields in $M^\delta_{G_\delta}$, 
\[
p^{\delta-1}X^{d-\ell p^\delta+1}Y^{\ell p^\delta-1}=p^{\delta-1}\epsilon_{\ell p^\delta-1}, 
\]
 And lemma \ref{lemma-vanishing-diagonal-action} shows that
$p^{\delta-1}X^{d-\ell p^2+1}Y^{\ell p^2-1}=0$ unless 
$p^{\delta-1}(p-1)\mid d-2\ell p^\delta+2$. Now let
\[
d=2\ell p^\delta-2+kp^{\delta-1}(p-1).
\]
Now 
\[
(T-\Id)(X^{d-\ell p^\delta-p^{\delta-1}+1}Y^{\ell p^\delta+p^{\delta-1}-1})=\sum_{i=0}^{\ell p^\delta+p^{\delta-1}-2}\binom{\ell p^\delta+p^{\delta}-1}{i} X^{d-i} Y^i.
\]
Note that for $\ell p^\delta-1<i\leq \ell p^\delta+p^{\delta-1}-2$, we have
$p^{\delta-1}(p-1)\nmid d-2i$, therefore $p^{\delta-1}X^{d-i}Y^i$ vanishes in 
$M^\delta_{G_\delta}$. Moreover, 
\begin{align*}
&(T-\Id)(X^{d-\ell p^\delta-p^{\delta-1}+1}Y^{\ell p^\delta+p^{\delta-1}-1})\\
&\equiv X^{d-\ell p^\delta-p^{\delta-1}+1}(X^{p^\delta}+Y^{p^\delta})^{\ell}(X+Y)^{p^{\delta-1}-1}\\
&-X^{d-\ell p^\delta-p^{\delta-1}+1}Y^{\ell p^\delta+p^{\delta}-1}\mod p
\end{align*}
which implies $\binom{\ell p^\delta+p^{\delta-1}-1}{i}\neq 0\mod p$ implies 
$i=i_1p^\delta+i_2$ with 
\[
0\leq i_1\leq \ell, \quad 0\leq i_2\leq p^{\delta-1}-1.
\]
In this case
\[
d-2i=2((l-i_1)p^\delta-i_2-1)+kp^{\delta-1}(p-1), 
\]
and $p^{\delta-1}(p-1)\mid d-2i$ implies $i_2=p^{\delta-1}-1$ and 
\[
i_1=
\ell-\frac{p+1}{2}\text{ or }\ell-1 
\]
As in the case of $\delta=2$, we show by induction that
\[
p^{\delta-1} X^{d-(\ell-1)p^\delta-(p^{\delta-1}-1)}Y^{(\ell-1)p^\delta+p^{\delta-1}-1}=0
\]
in $M^\delta_{G_\delta}$. The case of $\ell<\frac{p+1}{2}$ follows from
\begin{align*}
&p^{\delta-1}(T-\Id)(X^{d-\ell p^\delta-p^{\delta-1}+1}Y^{\ell p^\delta+p^{\delta-1}-1})\\
&=p^{\delta-1} a_\ell X^{d-(\ell-1)p^\delta-(p^{\delta-1}-1)}Y^{(\ell-1)p^\delta+p^{\delta-1}-1}
\end{align*}
for some unit $a_\ell\in (\bZ/p^\delta)^\times$. 
And for $\ell\geq \frac{p+1}{2}$, 
\begin{align*}
&p^{\delta-1}(T-\Id)(X^{d-\ell p^\delta-p^{\delta-1}+1}Y^{\ell p^\delta+p^{\delta-1}-1})\\
&=p^{\delta-1} a_\ell X^{d-(\ell-1)p^\delta-(p^{\delta-1}-1)}Y^{(\ell-1)p^\delta+p^{\delta-1}-1}\\
&+p^{\delta-1}b_{\ell}X^{d-(\ell-\frac{p+1}{2})p^\delta-(p^{\delta-1}-1)}Y^{(\ell-\frac{p+1}{2})p^\delta+p^{\delta-1}-1}
\end{align*}
for some unit $a_\ell, b_\ell\in (\bZ/p^\delta)^\times$. And our induction gives 
\[
p^{\delta-1}  X^{d-(\ell-1)p^\delta-(p^{\delta-1}-1)}Y^{(\ell-1)p^\delta+p^{\delta-1}-1}=0
\]
This shows
\begin{align*}
&p^{\delta-1}X^{(\ell-1)p^\delta+p^{\delta-1}-1}Y^{d-(\ell-1)p^\delta-(p^{\delta-1}-1)}\\
&=(S-\Id)(p^{\delta-1}X^{d-(\ell-1)p^\delta-(p^{\delta-1}-1)}Y^{(\ell-1)p^\delta+p^{\delta-1}-1})\\
&+p^{\delta-1}X^{d-(\ell-1)p^\delta-(p^{\delta}-1)}Y^{(\ell-1)p^\delta+p^{\delta-1}-1}\\
&=0
\end{align*}
We deduce from the condition
\[
d=2\ell p^\delta-2+kp^{\delta-1}(p-1)<p^\delta+p^{\delta-1}-2
\]
that
\[
k<p
\]
and
\[
p> 2\ell+k-\frac{k+1}{p}\geq \ell+k+\ell-\frac{k+1}{p}\geq \ell+k
\]
And by Lemma \ref{lemma-Lucas-theorem}, we have for $\ell+k<p$, 
\begin{align}\label{equation-binomial-unit-modulo-p-power}
&\binom{(k+\ell+1)p^{\delta-1}-1}{\ell p^{\delta-1}-1}\equiv \binom{k+\ell}{\ell-1},  \quad \mod p \quad (\text{for } p >3),
\end{align}

which shows that $\stirlingii{(\ell+r) p^\delta-rp^{\delta-1}-1}{\ell p^\delta-1}$ is a unit in $\bZ/p^\delta$.
Now 
\begin{align*}
&X^{(\ell-1)p^\delta+p^{\delta-1}-1}Y^{d-(\ell-1)p^\delta-(p^{\delta-1}-1)}\\
&=\sum_{r\leq d-(\ell-1)p^\delta-(p^{\delta-1}-1)}\stirlingii{d-(\ell-1)p^\delta-(p^{\delta-1}-1)}{r}\epsilon_r.
\end{align*}
and 
\[
d-(\ell-1)p^\delta-(p^{\delta-1}-1)=\ell p^\delta-1+(k+1)p^{\delta-1}(p-1).
\]
Applying (\ref{equation-stirling-number-p-power-vanishing-differ-r}) and (\ref{equation-stirling-number-ppower-nonvanishing-higher}) gives in $M^\delta_{G_\delta}$, 
\begin{align}\label{equation-monomial-muliplication-p-power-vanishing}
&p^{\delta-1}X^{(\ell-1)p^\delta+p^{\delta-1}-1}Y^{d-(\ell-1)p^\delta-(p^{\delta-1}-1)}\nonumber\\
&=p^{\delta-1}\epsilon_{\ell p^\delta-1+(k+1)p^{\delta-1}(p-1)}+p^{\delta-1}\binom{k+\ell}{\ell-1}\epsilon_{\ell p^\delta-1}.
\end{align}
The fact that $\ell+k<p$ shows that
$k<p-1$, and
\[
p^\delta\nmid \ell p^\delta+(k+1)p^{\delta-1}(p-1)
\]
hence 
\[
p^{\delta-1}\epsilon_{\ell p^\delta-1+(k+1)p^{\delta-1}(p-1)}=0.
\]
This shows $p^{\delta-1}\epsilon_{\ell p^\delta-1}=0$ in $M^\delta_{G_\delta}$. To finish the proof of (1), 
we are left to consider the case
$\ell=p$ (the condition $\ell p^\delta<p^{\delta+1}+p^{\delta-1}-2$ implies $\ell\leq p$). The equation (\ref{equation-stirling-power-series-p-power}) gives for $p$ odd, 
\[
\stirlingii{p^{\delta+1}-1}{\ell' p^\delta-1}=\left\{\begin{array}{lcr}
\binom{2p^\delta-1}{p^\delta-1}, &\text{ if } \ell'=1,  \\
0, &\text{ for }1<\ell'<p.
\end{array}\right.
\]
And Lemma \ref{lemma-Lucas-theorem} (Lucas's theorem) shows that
\[
\binom{2p^\delta-1}{p^\delta-1}=\binom{2p-1}{p-1}\equiv 0, 
\quad \mod p.
\]
So we still have
\[
p^{\delta-1}X^{d- p^{\delta+1}+1}Y^{ p^{\delta+1}-1}=p^{\delta-1}\epsilon_{ p^{\delta+1}-1}.
\]
And the same argument as the case $\ell<p$ holds as long as
we have $p+k<p$, i.e, k<0(which is the case since we suppose $d<p^{\delta+1}+p^{\delta-1}-2$). 

This finishes the proof of (1).

As for $d=p^{\delta+1}+p^{\delta-1}-2$,  again, we need to consider the element 
$\epsilon_{\ell p^\delta-1}$ with $1\leq \ell\leq p$. First
consider the case $1\leq \ell<p$, then as before, we still have
\[
p^{\delta-1}X^{d-\ell p^\delta+1}Y^{\ell p^\delta-1}=p^{\delta-1}\epsilon_{\ell p^\delta-1}, 
\]
Furthermore, 
\[
p^{\delta-1}(p-1)\mid d-2\ell p^\delta+2
\]
implies for $p>3$
\[
\ell=
1\text{ or } \frac{p+1}{2}.
\]
Therefore we are reduce to consider the case
$\ell=1, \frac{p+1}{2}, p$. Note that 
\[
d=p^\delta-1+p^{\delta-1}(p-1)^2.
\]
Then for $\ell=1$ and $k=p-1$, equation (\ref{equation-binomial-unit-modulo-p-power}) becomes
\[
\binom{(k+\ell+1)p^{\delta-1}-1}{\ell p^{\delta-1}-1}\equiv 1,  \quad \mod p
\]
which is still a unit.
And (\ref{equation-monomial-muliplication-p-power-vanishing}) becomes
\[
0=p^{\delta-1}X^{p^{\delta-1}-1}Y^{p^{\delta+1}-1}=p^{\delta-1}\epsilon_{\ell p^{\delta+1}-1}+p^{\delta-1}\epsilon_{ p^\delta-1}.
\]
But we know that 
\[
X^{p^{\delta+1}-p^\delta+p^{\delta-1}-1}Y^{p^\delta-1}=\epsilon_{p^\delta-1}+\sum_{r<p^\delta-1}\stirlingii{p^\delta-1}{r} \epsilon_{r}
\]
is primitive, hence so is $\epsilon_{p^{\delta+1}-1}$. 
And we conclude that $\epsilon_{\ell p^{\delta+1}-1}+\epsilon_{ p^\delta-1}$ is of order at most
$p^{\delta-1}$.  Finally for $\ell=\frac{p+1}{2}$, the (\ref{equation-monomial-muliplication-p-power-vanishing}) becomes
\[
0=pX^{\frac{p^\delta(p-1)}{2}+p^{\delta-1}-1}Y^{\frac{p^\delta(p+1)}{2}-1}=p\epsilon_{\frac{p^\delta(p+1)}{2}-1}.
\]

This finishes the proof of (2).

As for (3), we consider the action of 
\[
M^{\delta, G_\delta}_{\prim}=\bZ/p^\delta[f_{1, \delta}, f_{2, \delta}].
\]

We first show the following lemma
\begin{lemma}\label{lemma-quotient-vanishing-higer-degree-p-power}
We have
\[
\oplus_{d>p^{\delta+1}+p^{\delta-1}-2}(M^1_{G_1}/(f_{1,\delta},f_{2,\delta})M^1_{G_1})_d=0
\]
\end{lemma}

\begin{proof}[Proof of lemma \ref{lemma-quotient-vanishing-higer-degree-p-power}]
By definition, we have
\[
f_{1,\delta}=f_1^{p^{\delta-1}}, f_{2, \delta}=f_{2}^{p^{\delta-1}}.
\]
By Proposition \ref{prop-divided-power-struture-modp}, we 
know that 
 $M^1_{G_1}/(f_{1,\delta},f_{2,\delta})M^1_{G_1}$ 
 as $\bF_p$-vector space, is generated by 
\[
f_2^{a_0}, f_2^{a_1}X^{p-1}Y^{p-1}, f_2^{a_2}X^{2(p-1)}Y^{p-1}, \cdots,f_2^{a_{p-2}}X^{(p-2)(p-1)}Y^{p-1},
f_1^bf_2^{a_{p}}X^{p(p-1)}Y^{p-1}
\]
with $0\leq a_i\leq p^{\delta-1}-1$ and $0\leq b\leq p^{\delta-1}-1$.
All of these generators are of degree $\leq p^{\delta+1}+p^{\delta-1}-2$.
\end{proof}
Now (3) is a consequence of Lemma \ref{lemma-quotient-vanishing-higer-degree-p-power}. We are done.
\end{proof}

\begin{proof}[Proof of Proposition \ref{prop-second-cohomology-module-structures}]
Now Proposition \ref{prop-second-cohomology-module-structures} is 
proved combining Corollary \ref{cor-second-cohomology-nonvanishing-p-powers-freeness} and Lemma \ref{lemma-primitive-element-non-existence-higer-powers}.
\end{proof}

\section{Application to Congruence of Modular Forms}\label{sec-cmf}

Using results from previous sections, we
are able to determine the torsions 
of $H^1(X, \tilde{\cM}_n)$ and $H^2_c(X, \tilde{\cM}_n)$. Note that though the 
2 and 3 torsions are not determined by the article, 
we still list them in the examples.

\begin{definition}
We say a prime $\ell$ is good with respect to 
$n$ if $3<\ell<n$ and both $H^1(X, \Tilde{\cM}_n)_{\tor}$ and $ H^2_c(X, \Tilde{\cM}_n)$ contain no $\ell$-power
torsions. Let $\bT(n)$ denotes the set of good primes
with respect to $n$.
\end{definition}

\begin{example}
We have
\[
H^1(X, \tilde{\cM}_{10})_{\tor}=\bZ/4;
\]
which is already known to Harder. And
\[
H^2_c(X, \tilde{\cM}_{10})=(\bZ/2)^{\oplus 2}\oplus \bZ/3.
\]
Moreover, 
\[
H^1(X, \tilde{\cM}_{22})_{\tor}=(\bZ/4)^{\oplus 2}\oplus \bZ/2\oplus \bZ/3
\]
\[
H^2_c(X, \tilde{\cM}_{22})=(\bZ/2)^{\oplus 3}\oplus \bZ/4\oplus (\bZ/3)^{\oplus 2}.
\]
\end{example}

Now recall that from (\ref{eq1}) we have the following exact
sequence
\begin{align*}
0&\rightarrow H^1(X, \Tilde{\cM}_n)_{\int, !}/H^1_{!}(X, \Tilde{\cM}_n)_{\int}\rightarrow
H^1(\partial X, \Tilde{\cM}_n)_{\tor}/H^1(X, \Tilde{\cM}_n)_{\tor}\rightarrow \nonumber\\
&\rightarrow H^2_c(X, \Tilde{\cM}_n)\rightarrow 0.
\end{align*}
Knowing the group $H^1(\partial X, \Tilde{\cM}_n)_{\tor}/H^1(X, \Tilde{\cM})_{\tor}$ would allow us to draw information 
about the map
\[
H^1(X, \Tilde{\cM}_n)_{\int, !}/H^1_{!}(X, \Tilde{\cM}_n)_{\int}\rightarrow
H^1(\partial X, \Tilde{\cM}_n)_{\tor}/H^1(X, \Tilde{\cM}_n)_{\tor}
\]
which gives us congruence of cuspidal forms to Eisenstein 
series modulo $p$-powers.

\begin{example}
For case $n=10$, we know the $5$ and $7$ torsions appear in 
$H^1(\partial X, \Tilde{\cM}_n)_{\tor}$ but not in 
$H^1(X, \Tilde{\cM}_n)_{\tor}$ or $ H^2_c(X, \Tilde{\cM}_n)$. 
Therefore applying Proposition \ref{prop-Hecke-operator-boudary-eigenvalue}
allows us to recover the famous conguences
\[
\tau(p)\equiv p^5+p^6(\equiv p+p^2),\quad \mod 5, 
\]
and
\[
\tau(p)\equiv p^7+p^4(\equiv p+p^4), \quad \mod 7, 
\]
where $\tau$ is the Ramanujan-$\tau$ function.
As for the case of $n=22$, the Hecke eigenform is defined
over the number field
\[
K:=\bQ[\alpha]/(\alpha^2 - \alpha - 36042).
\]
And we get an eigenform
\begin{align*}
f(q)=q+(-24\alpha + 552) q^2+(1152\alpha + 169164)q^2+(-25920\alpha + 12676288)q^3+ 
\end{align*}
Note that the primes $\ell=5, 7, 11, 13, 17, 19$ are all good primes, 
i.e, we get a congruence of Hecke eigenform to Eisenstein series. For $\ell=5, 7, 11$, we have several $\ell$-torsion classes 
in $H^1(\partial X, \Tilde{\cM}_n)_{\tor}$, hence we 
get several congruences. Consider the case
$\ell=5$, let $\ell=\mathfrak{l}_1\mathfrak{l}_2$ in $K$
such that $\alpha=2\mod \mathfrak{l}_1$ and
$\alpha=4\mod \mathfrak{l}_2$. Then 
for $f(q)=\sum a_i q^i$,
\[
a_p\equiv p^{15}+p^8\equiv p^{20}+p^3, \quad \mod \mathfrak{l}_1
\]
\[
a_p\equiv p^5+p^{18}\equiv p^{10}+p^{13}, \quad \mod \mathfrak{l}_2.
\]
And for $\ell=13, 17, 19$, consider the
case $\ell=13$, it splits into two primes
in $K$. In fact, let
$\ell=\mathfrak{l}_3\mathfrak{l}_4$ such
that $\alpha=3\mod \mathfrak{l}_1$. Then we get
for $f(q)=\sum a_i q^i$, 
\[
a_p=p^{13}+p^{10}, \quad \mod \mathfrak{l}_3.
\]
And we get no congruence modulo $\mathfrak{l}_4$.
\end{example}

\begin{example}
The reader familiar with classical results on 
congruence of Ramanujan-$\tau$ may observe that
we only get congruence modulo $5$ in the above example, 
but indeed we have
\[
\tau(p)\equiv p+p^{10}, \quad \mod 25.
\]
This could be explained as follows.
The Hecke-module
$$H^1(X, \Tilde{\cM}_{50})_{\int, !}\otimes_{\bZ} \bZ_5$$
contains an eigenform
\[
f=\sum_{i=1}^{\infty}a_i q^i,
\]
with
\[
a_p\equiv p+p^{10}, \quad \mod 25.
\]
This is explained by the fact that the image of 
$f$ under the natural map 
\[
H^1(X, \Tilde{\cM}_n)_{\int, !}/H^1_{!}(X, \Tilde{\cM}_n)_{\int}\rightarrow
H^1(\partial X, \Tilde{\cM}_n)_{\tor}/H^1(X, \Tilde{\cM}_n)_{\tor}
\]
is a $25$-torsion( as we said before, we leave the determination of the image of the above map for another paper). 
Finally, a classical result of Serre(cf. \cite{Ser73} \S 1.3) tells us that we have a congruence
\[
\Delta\equiv f, \quad \mod 25, 
\]
where $\Delta=\sum_{n=1}^{\infty} \tau(n)q^n$
is the modular form of weight 12.
\end{example}

\begin{example}
Our final example concerns the case of $n=34$, we have
\begin{align*}
&H^1(X, \tilde{\cM}_{34})_{\tor}=(\bZ/4)^{\oplus 3}\oplus (\bZ/3)^{\oplus 2}\oplus (\bZ/2)^{\oplus 2}, \\    
 &H^2_c(X, \tilde{\cM}_{34})=\bZ/8\oplus \bZ/4\oplus (\bZ/3)^{\oplus 2}\oplus (\bZ/2)^{\oplus 4}.  
\end{align*}
In particular, we know that $5\in\bT(34)$.  
And we find a Hecke eigenform $f=\sum_{i=1}^{\infty}a_i q^i\in \bZ_5[[q]]$ with 
\[
a_p\equiv p^{25}+p^{10}, \quad \mod 25, 
\]
for all prime $p$.
\end{example}

We have the following general results 
\begin{teo}
Let $n>0$ be even. Then for $\ell\in \bT(n)$, 
any $\ell$-torsion class gives rise to a congruence between some cuspidal form of level one
and Eisenstein series modulo $\ell$.
\end{teo}

\remk Sometimes even the primes 
that are not good with respect to $n$ contribute 
to congruence. But this requires the determination 
of the image of $H^1(\partial X, \Tilde{\cM}_n)_{\tor}$ in 
$ H^2_c(X, \Tilde{\cM}_n)$. We leave this
to the next paper.

\remk It remains to determine the Hecke module stucture on $H^1(X, \Tilde{\cM}_n)_{\tor}$ and $ H^2_c(X, \Tilde{\cM}_n)$. One could also attach Galois representations to the torsion classes we constructed. We plan to return to these questions in the next paper.

\section{Stirling Numbers of the Second Kind}

\begin{teo}\label{thm-stirl-valuation}
Assume $p>3$ be prime. Let $p$ be prime and $\delta>0$. We have 
\[
\val_p(\stirlingii{n}{p^{\delta}-1})\geq \delta-1-\val_p(n+1).
\]
\end{teo}

We recall the following results concerning Stirling number of 
the second kind

\begin{prop}(cf. \cite{CH10} Theorem 5.2)\label{prop-stirl-ppower-cong}
Let $p$ be odd and $m\geq 1, n\geq p^m$. Then we have
\[
\stirlingii{n}{p^m}\equiv \left\{\begin{array}{lcr}
\binom{\frac{n-p^{m-1}}{p-1}-1}{\frac{n-p^{m}}{p-1}}\mod p^m,&  \text{ if }
n=1 \mod p-1, \\
0 \mod p^m ,&  \text{ otherwise }.
\end{array}\right. 
\]
\end{prop}

\begin{lemma}\label{lem-stirl-congru-equal}
We have
\[
\stirlingii{n-1}{p^m-1}\equiv \stirlingii{n}{p^m}, \quad \mod p^m
\]
\end{lemma}

\begin{proof}
This follows from the fact that 
\[
\sum_{n=p^{m}}^{\infty}\stirlingii{n}{p^m}X^n=\frac{X^{p^m}}{(1-X)(1-2X)\cdots (1-p^mX)}
\]
and the right hand side equals to 
\[
X\frac{X^{p^m-1}}{(1-X)(1-2X)\cdots (1-(p^m-1)X)}, \quad \mod p^m.
\]
\end{proof}

We also need the following property of binomial coefficients.
\begin{lemma}\label{lem-binom-valuation}
We have
\[
\val_p(\binom{n-1+p^{m}}{n})\geq m-\val_p(n),
\]
and for $1\leq n\leq p^m$, 
\[
\val_p(\binom{p^{m}}{n})\geq m-\val_p(n).
\]
\end{lemma}

\begin{proof}
Let $m\geq 1$ and $n\geq 1$.
We have
\begin{align*}
\val_p(\binom{n-1+p^{m}}{n})&=\val_p(\frac{(n-1+p^m)(n-2+p^m)\cdots (p^m+1)p^m}{n!})\\
&=m-\val_p(n)+\sum_{1<i<n}(\val_p(p^m+i)-\val_p(i)).
\end{align*}
Note that 
we have
\[
\val_p(p^m+i)-\val_p(i)=0, \text{ if } \val_p(i)<m.
\]
Therefore, 
\begin{align*}
&\sum_{1<i<n}(\val_p(p^m+i)-\val_p(i))\\
&=\sum_{0<j<n/p^m}(\val_p(p^m(1+j))-\val_p(p^mj))\\
&=\sum_{0<j<n/p^m}(\val_p(1+j)-\val_p(j))\\
&=\val_p(j_{\max})\geq 0, 
\end{align*}
where $j_{\max}$ is the maximal integer satisfying $0<j<n/p^m$.
Similarly, 
\begin{align*}
\val_p(\binom{p^{m}}{n})&=\val_p(\frac{p^m(p^m-1)(p^m-2)\cdots (p^m-n+1)}{n!})\\
&=m+\sum_{2\leq i\leq n}(\val_p(p^m-i+1)-\val_p(i)).
\end{align*}
And
\[
\val_p(p^m-i+1)=\val_p(i-1), \text{ for } 2\leq i\leq p^m.
\]
This shows that
\[
\val_p(\binom{p^{m}}{n})\geq m-\val_p(n).
\]
\end{proof}
\remk We thank Robin Bartlett for helping us 
with the proof of the lemma.

\begin{proof}[Proof of Theorem \ref{thm-stirl-valuation}]
According to Lemma \ref{lem-stirl-congru-equal}, 
we know 
\[
\stirlingii{n}{p^{\delta}-1}\equiv \stirlingii{n+1}{p^{\delta}}\quad \mod p^{\delta}.
\]
First of all, we assume that $p$ is odd. 
Then applying Proposition \ref{prop-stirl-ppower-cong}, we known that $\stirlingii{n+1}{p^{\delta}}\equiv 0 \mod p^{\delta}$
if $n\not\equiv 0\mod p-1$, from which we deduce
\[
\val_p(\stirlingii{n+1}{p^{\delta}})\geq \delta.
\]
If $n\equiv 0\mod (p-1)$, let $n=a(p-1)+p^{\delta}-1$. Then by the same proposition, we get
\[
\stirlingii{n+1}{p^{\delta}}\equiv \binom{\frac{n+1-p^{\delta-1}}{p-1}-1}{\frac{n+1-p^{\delta}}{p-1}}\equiv \binom{a+p^{\delta-1}-1}{a}\mod p^{\delta}.
\]
Applying Lemma \ref{lem-binom-valuation}, we know
\[
\val_p(\binom{a+p^{\delta-1}-1}{a})\geq \delta-1-\val_p(a).
\]
If $\val_p(a)< \delta$, we know further that 
\[
\val_p(n+1)=\val_p(a(p-1)+p^{\delta})=\val_p(a).
\]
If $\val_p(a)\geq \delta$
\[
\val_p(n+1)\geq \delta.
\]
This shows that 
\[
\stirlingii{n}{p^{\delta}-1}\geq \delta-1-\val_p(n+1).
\]

\end{proof}

We finish this section with the following proposition
\begin{prop}\label{prop-binomial-expansion-valuation}
Let $\gamma\geq 1$ and $p$ is prime. We have
\[
\val_p((1-p^{\gamma})^j-1)\geq \val_p(j)+\gamma.
\]
The equality holds whenever $p>2$ or $p=2$
and $\gamma\geq 2$.
\end{prop}

\remk We learn the proof from 
Carlo Pagano.

\begin{proof}
We assume that either $p$ is odd and $i_0=1$ or $p=2$ and $i_0=2$.
Consider the following filtration of subgroups on $\bZ_p$
\[
1+p^{i_0}\bZ_p:=U_1\supseteq U_2\supseteq \cdots 
\]
with $U_i=1+p^{i_0+i}\bZ_p$. Then 
$\bZ_p$ acts on $U_i$ via taking powers, i.e, for
$s\in \bZ_p$ and $a\in U_1$, 
\[
\phi_s(a)=a^s.
\]
Note that $\phi_s$ satisfies the following properties
\begin{description}
\item[(1)]$\phi_{s_1}(a)\phi_{s_2}(a)=\phi_{s_1+s_2}(a)$, 
\item[(2)]$\phi_{s_1}(\phi_{s_2}(a))=\phi_{s_1s_2}(a)$.
\end{description}
We claim that $\phi_s(U_i)=U_{i+\val_p(s)}$, 
in particular, $\phi_s$ is an automorphism of $U_1$ if $\val_p(s)=0$. 
By property $(1)$ and $(2)$, we only need 
to check that 
$\phi_s$ is an automorphism of $U_1$ for
$s=1, \cdots, p-1$ and $\phi_p(U_i)=U_{i+1}$.
Indeed, for $1\leq i, j, k\leq p-1$,  
\[
(1+jp^{i})^k=1+kjp^{i}+\cdots\in U_i
\]
and 
\[
(1+jp^{i})^{p}=1+jp^{i+1}+\sum_{2\leq \ell\leq p}\binom{p}{\ell}j^{\ell}p^{\ell i}.
\]
By Lemma \ref{lemma-Lucas-theorem}, we know that
\[
\val_p(\binom{p}{\ell})\geq 1, \ell=1, 2, \cdots, p-1.
\]
Therefore, we have for $1\leq j\leq p-1$, 
\[
\val_p((1+jp^{i})^p-1)=i+1.
\]

\end{proof}

\bibliographystyle{plain}
\bibliography{biblio}

\end{document}